\documentclass[a4paper,10pt]{amsart}
\usepackage[top=80pt,bottom=65pt,left=70pt,right=70pt]{geometry}
\usepackage{marginnote}
\normalfont

\usepackage{pst-all, pstricks, pst-node}
\usepackage{subfigure}
\usepackage{graphicx, amssymb, amsmath, amscd, hhline}
\usepackage[arrow,tips,matrix]{xy}
\usepackage{auto-pst-pdf}
\usepackage{float}

\usepackage{color}
\usepackage{hyperref}
\hypersetup{colorlinks}
\hypersetup{
    bookmarks=true,         
    unicode=false,          
    pdftoolbar=true,        
    pdfmenubar=true,        
    pdffitwindow=false,     
    pdfstartview={FitH},    
    pdftitle={My title},    
    pdfauthor={Author},     
    pdfsubject={Subject},   
    pdfcreator={Creator},   
    pdfproducer={Producer}, 
    pdfkeywords={keyword1} {key2} {key3}, 
    pdfnewwindow=true,      
    colorlinks=true,       
    linkcolor=blue,          
    citecolor=red,        
    filecolor=violet,      
    urlcolor=violet       
}

\theoremstyle{plain}
\newtheorem{theorem}{Theorem}[section]
\newtheorem{proposition}[theorem]{Proposition}
\newtheorem{lemma}[theorem]{Lemma}
\newtheorem*{main theorem}{Theorem}
\newtheorem{question}[theorem]{Question}
\newtheorem{corollary}[theorem]{Corollary}

\theoremstyle{definition}
\newtheorem*{acknowledgement}{Acknowledgement}
\newtheorem*{organization}{Organization}
\newtheorem*{concluding remark}{Concluding remark}
\newtheorem{definition}[theorem]{Definition}
\newtheorem{remark}[theorem]{Remark}
\newtheorem{notation}[theorem]{Notation}
\newtheorem{example}[theorem]{Example}

\numberwithin{table}{section} \numberwithin{figure}{section}
\numberwithin{equation}{section}

\DeclareMathOperator{\Map}{Map}

\DeclareMathOperator{\sym}{S}

\DeclareMathOperator{\HR}{HR}

\DeclareMathOperator{\supp}{supp \:}
\DeclareMathOperator{\vol}{vol}
\DeclareMathOperator{\Conv}{Conv}

\newcommand{\field}[1]{\mathbb{#1}}
\newcommand{\C}{\field{C}}
\newcommand{\R}{\field{R}}
\newcommand{\Z}{\field{Z}}

\begin{document}

\title[Hard Lefschetz Property for Hamiltonian torus actions]{Hard Lefschetz Property for Hamiltonian torus actions on 6-dimensional GKM manifolds}

\author{Yunhyung Cho}
\address{Department of Mathematics Education, Sungkyunkwan University, 
25-2, Sungkyunkwan-ro, Jongno-gu, Seoul, 03063, Republic of Korea.}
\email{yunhyung@skku.edu}

\author[M. K. Kim]{Min Kyu Kim}
\address{Department of Mathematics Education,
Gyeongin National University of Education, 45 Gyodae-Gil,
Gyeyang-gu, Incheon, 407-753, Republic of Korea}
\email{mkkim@kias.re.kr}

\date{\today}

%

\begin{abstract}
Let
$(M,\omega)$ be a 6-dimensional closed symplectic
manifold with a Hamiltonian $T^2$-action. We show
that if the action is GKM and its GKM graph is
index-increasing, then $(M,\omega)$ satisfies the
hard Lefschetz property.
\end{abstract}

\maketitle
\setcounter{tocdepth}{1}
\tableofcontents

\section{Introduction}
\label{secIntroduction}

Let $T$ be a compact torus acting effectively on a closed symplectic manifold $(M,\omega)$ in a Hamiltonian fashion.
If the $T$-action is GKM, the celebrated theorem \cite[Theorem 1.2.2]{GKM} due to Goresky-Kottwitz-MacPherson
tells us that the equivariant cohomology ring of $M$ is completely determined by the corresponding GKM graph, 
which is an image of zero and one-dimensional torus orbits in $M$ under a moment map. In particular, since the ordinary cohomology ring $H^*(M;\R)$ of $M$ can be obtained 
from its equivariant cohomology ring by extension of scalars, the product structure of $H^*(M;\R)$  is determined 
by the GKM graph. 

In this paper, we study the hard Lefschetz property of a closed Hamiltonian GKM manifold.
We say that a closed symplectic manifold $(M,\omega)$ satisfies the {\em hard Lefschetz property} if 
\begin{displaymath}
\begin{array}{cccc}
\wedge [\omega]^{n-l} : & H^l(M;\R) & \longrightarrow & H^{2n-l}(M;\R)\\[0.5em]
& \alpha & \longmapsto & \alpha \wedge [\omega]^{n-l} \\[0.5em]
\end{array}
\end{displaymath}
is an isomorphism for every $l=0, 1, \cdots, n$. 
It is clear that the product structure of $H^*(M;\R)$ and the cohomology class $[\omega] \in H^2(M;\R)$ determine whether $(M,\omega)$ satisfies the
hard Lefschetz property or not, and therefore it is natural to ask how to check the hard Lefschetz property of a closed Hamiltonian GKM manifold by ``looking up'' the corresponding GKM graph.
 
It is known that the hard Lefschetz property does not hold 
in general. See \cite{Cho1} or \cite{Go} for example. However, it is not known whether $(M,\omega)$ satisfies the hard Lefschetz property when 
$(M,\omega)$ admits a Hamiltonian torus action with isolated fixed points. Our work is motivated by the following question posed by Karshon. 

\begin{question} \label{question_Karshon}\cite{JHKLM}
Let $(M,\omega)$ be a closed symplectic manifold with an effective Hamiltonian circle action. Assume that all fixed
points are isolated. Then, does $(M,\omega)$ satisfy the hard Lefschetz property?
\end{question}

Note that if $(M,\omega)$ satisfies the hard Lefschetz property, then we can easily see that 
the sequence $\{b_0(M), b_2(M), \cdots, b_{2n}(M)\}$ is unimodal\footnote{A sequence of real numbers $a_1, \cdots, a_n$ is called {\em unimodal} if there exists an integer $k\geq 1$ such that $a_1\leq \cdots \leq a_k \geq \cdots \geq a_{n}$} 
where $b_i(M)$ denotes the $i$-th Betti number of $M$. This leads to
the following question, posed by Tolman, regarded as a weak version of Question \ref{question_Karshon}.

\begin{question}\label{question_unimodality}\cite{JHKLM}
    Let $(M,\omega)$ be a closed symplectic manifold with an effective Hamiltonian circle action. If all fixed points are isolated, then 
    is the sequence $\{b_0(M), b_2(M), \cdots, b_{2n}(M)\}$ unimodal?
\end{question}

Following a remark by Karshon in \cite{JHKLM}, we observe that 
the condition of ``admitting isolated fixed points'' is a strong assumption in the sense that 
an example of a closed symplectic non-K\"{a}hler Hamiltonian $S^1$-manifold with isolated fixed points has not been found so far.
In fact, there are several positive results on Question \ref{question_Karshon} and Question \ref{question_unimodality}
for a Hamiltonian torus action with 
isolated fixed points. For example, Delzant \cite{De} proved that every closed symplectic toric manifold
is K\"{a}hler and hence the hard Lefschetz property holds.
Also, Karshon \cite{Ka} proved that any four dimensional closed Hamiltonian
$S^1$-manifold $(M,\omega)$ with isolated fixed points admits an $S^1$-invariant K\"{a}hler form. In this case, the hard Lefschetz property is rather obvious
since $H^1(M ; \R) = H^3(M ; \R) = 0$ by the Frankel's theorem \cite[Corollary 2]{Fr}.
Also, some positive answers {\blue to} Question \ref{question_Karshon} and Question \ref{question_unimodality} are provided
in \cite{Cho2}, \cite{CK1}, \cite{CK2}, and \cite{Lu} under certain technical assumptions. 

Throughout this paper we restrict our attention to Question \ref{question_Karshon} for closed Hamiltonian GKM manifolds.
Note that $(M,\omega)$ satisfies the hard Lefschetz property if and only if the Hodge-Riemann bilinear form defined as
\begin{displaymath}
            \begin{array}{cccc}
                \HR_l : & H^l(M)  \times  H^l(M) & \longrightarrow & \R \\[0.5em]
                 & (\alpha, \beta) & \longmapsto &  <\alpha \beta [\omega]^{n-l}, [M]>\\[0.5em]
            \end{array}
        \end{displaymath}
is non-degenerate for every $l=0,1,\cdots,n$. 
To check the non-degeneracy of $\mathrm{HR}_l$, 
we first consider certain two bases $\mathcal{B}_l^+$ and $\mathcal{B}_l^-$ of $H^l(M;\R)$,  
which consist of so-called the {\em equivariant Thom classes}  in $H^l_T(M;\R)$ introduced by Guillemin-Zara \cite{GZ}. (See also Section \ref{secGraphCohomologyAndEquivariantCohomologyOfHamiltonianGKMManifolds}.)
Then we show that the matrix, denoted by $A_l(M,\omega)$, representing $\mathrm{HR}_l$ with respect to the pair ($\mathcal{B}_l^+, \mathcal{B}_l^-)$ 
is obtained from the GKM graph by using the ABBV-localization theorem and Goldin-Tolman's theorem \cite{GT}.
(See Proposition \ref{proposition_coeff_higher_dim} for the detail.) 
Also, in case of $n-l = 1$, we show that $A_l(M,\omega)$ has many zero entries. 
(See Corollary \ref{corollary_index_differ_by_two}.) Furthermore, we prove the following if $M$ is of dimension six.

\begin{theorem} \label{theorem_main}
Let $(M, \omega)$ be a 6-dimensional closed symplectic manifold
equipped with an effective Hamiltonian $T^2$-action.
If the action is GKM and the corresponding GKM graph is index increasing\footnote{See Definition \ref{definition_index_increasing}.}, then $(M,\omega)$ satisfies
the hard Lefschetz property.
\end{theorem}

\begin{example}\label{example_tolman}
    In \cite{T}, Tolman constructed a six-dimensional closed Hamiltonian GKM manifold $(M,\omega)$ which has no
    K\"{a}hler metric invariant under the action. The corresponding GKM graph is given in Figure \ref{picture: woodward}.

\begin{figure}
\begin{center}
\begin{pspicture}(-1,-1)(5,5.5) \footnotesize
\pspolygon[fillstyle=solid,fillcolor=lightgray,
linewidth=1.5pt](1,0)(0,1)(0,4)(4,0)(1,0)

\psline[linewidth=0.5pt,arrowsize=5pt]{->}(0,-0.5)(0,5)
\psline[linewidth=0.5pt,arrowsize=5pt]{->}(-0.5,0)(5,0)
\psline[arrowsize=5pt]{->}(3,2)(2,3)
\psline[linewidth=1.5pt](1,0)(1,2)(2,1)(4,0)
\psline[linewidth=1.5pt](0,1)(2,1)(1,2)(0,4)

\psdots[dotsize=5pt](1,0)(0,1)(2,1)(1,2)(4,0)(0,4)

\uput[dl](0,0){$O$} \uput[r](5,0){$x$}
\uput[u](0,5){$y$} \uput[d](1,0){$1$}
\uput[l](0,1){$1$} \uput[d](4,0){$4$}
\uput[l](0,4){$4$} \uput[ur](2.5,2.5){$\xi$}

\end{pspicture}
\end{center}
\caption{ \label{picture: woodward} Tolman's
Hamiltonian GKM manifold}
\end{figure}
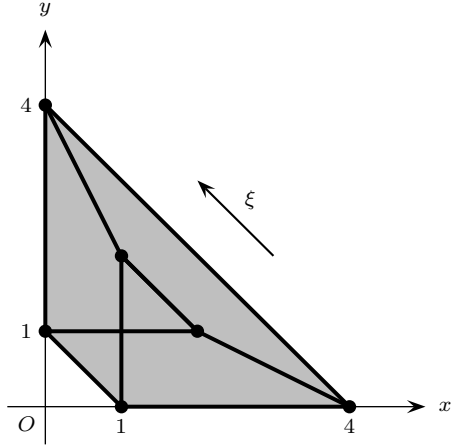
With respect to the $\xi$ described in Figure \ref{picture: woodward}, we see that the GKM graph
 is index-increasing so that Tolman's manifold satisfies the
hard Lefschetz property by Theorem \ref{theorem_main}.
In fact, Woodward already pointed out in \cite[page 9]{Wo2} that Tolman's manifold satisfies the hard Lefschetz property with a hint for a proof, which
seems to rely on computation of the cohomology ring of $M$. Also, he constructed more examples of non-K\"{a}hler GKM-manifolds using $U(2)$-equivariant surgery
and they have the same x-ray with Tolman's example \cite[Proposition 3.6]{Wo},  
and therefore their GKM graphs are all index-increasing. Consequently, every
Woodward's example satisfies the hard Lefschetz property by Theorem \ref{theorem_main}.
\end{example}

\begin{organization}
In Section
\ref{secEquivariantCohomology}, we give a brief
introduction to the equivariant cohomology theory for Hamiltonian
torus actions and recall the ABBV-localization theorem which will be used in order to compute the 
matrix $A_l(M,\omega)$ representing the Hodge-Riemann bilinear form $\mathrm{HR}_l$.
In Section
\ref{secGraphCohomologyAndEquivariantCohomologyOfHamiltonianGKMManifolds},
we provide some background on Hamiltonian GKM manifolds and their graph cohomology rings.
In Section \ref{secHodgeRiemannBilinearFormInHigherDimension}, 
we compute the matrix $A_l(M,\omega)$ by using combinatorial data of a GKM graph.
In Section
\ref{secSixDimensionalHamiltonianGKMManifoldsWithAnIndexIncreasingGraph},
we prove our main theorem (Theorem
\ref{theorem_main}). Finally, in Section
\ref{secProofOfPropositionRefPropositionGraphShapeAndRefPropositionConvexIf8},
we prove two propositions crucially used in Section
\ref{secSixDimensionalHamiltonianGKMManifoldsWithAnIndexIncreasingGraph}.

\end{organization}

\begin{acknowledgement}
The authors thank anonymous referees for their endurance
and kindness to improve the paper, especially to bring the beautiful papers \cite{GT}, \cite{ST}, and \cite{Mo} to our attention.
The first author was supported by the National Research Foundation of Korea(NRF) grant funded by the Korea government(MSIP; Ministry of Science, ICT \& Future Planning) (NRF-2017R1C1B5018168).
The second author is supported by GINUE research fund.
\end{acknowledgement}

\bigskip

\section{Equivariant cohomology}
\label{secEquivariantCohomology}

Throughout this paper, we assume that an action
of a Lie group on a manifold is {\em effective}, unless stated otherwise. Also, we
take cohomology with coefficients in $\R$.

Let $(M,\omega)$ be a closed symplectic manifold admitting Hamiltonian $T$-action where $T$ is a compact $m$-dimensional torus for some integer $m \geq 1$.
Then the {\em equivariant cohomology} of $M$ is defined by
\[
    H^*_T(M) =: H^*(M \times_T ET)
\]
where $ET$ is a contractible space on which $T$ acts freely.
In particular, the equivariant cohomology of a point is given by
$H^*_T(\mathrm{pt}) = H^*(\mathrm{pt} \times_T ET) = H^*(BT)$ where $BT = ET / T$ is the classifying space of $T$.
Note that if $T = S^1$, then $BS^1$ can be constructed as an inductive limit of the sequence of Hopf fibrations
                            \[
                                \begin{array}{ccccccccc}
                                    S^3          & \hookrightarrow & S^5        & \hookrightarrow & \cdots & S^{2n+1} & \cdots & \hookrightarrow & ES^1 \sim S^{\infty} \\
                                    \downarrow   &                 & \downarrow &                 & \cdots & \downarrow & \cdots &                 & \downarrow \\
                                    \C P^1       & \hookrightarrow & \C P^2     &\hookrightarrow  & \cdots & \C P^n      & \cdots & \hookrightarrow &BS^1 \sim \C P^{\infty}
                                \end{array}
                            \]
Thus we have $$H^*(BS^1) \cong \R[x]$$ where $x$ is an
element of degree two such that $\langle x, [\C P^1]
\rangle = 1$. Similarly, if we choose an ordered $\Z$-basis $\frak{X} =
\{X_1, \cdots, X_m\}$ for the lattice\footnote{The {\em lattice} of $\mathfrak{t}$ means the kernel of the exponential map from 
$\frak{t}$ to $T$.} in $\mathfrak{t}$ and a
decomposition $T = S^1 \times \cdots \times S^1$
corresponding to $\frak{X}$, then we can easily check that
$BT$ is homotopy equivalent to the $m$-times product
of $\C P^\infty$ and hence 
\begin{equation}\label{equation_polynomial_ring}
    H^*(BT) \cong \sym(\mathfrak{t}^*) = \R[x_1, \cdots, x_m]
\end{equation}
where $\sym(\mathfrak{t}^*)$ is the symmetric tensor
algebra of $\mathfrak{t}^*$ and each $x_i \in
\mathfrak{t}^*$ is the dual of $X_i$ and is of degree two for $i=1,\cdots,
m$.

\subsection{Equivariant formality} Note that a projection map $M \times ET \rightarrow ET$ on the second factor is $T$-equivariant so that
it induces the map
\[
     \pi : M \times_{T} ET \rightarrow BT
\]
which makes $M \times_T ET$ into an $M$-bundle over $BT$
\begin{equation}\label{diagram : M-bundle over BT}
           \begin{array}{ccc}
            M \times_{T} ET & \stackrel{f} \hookleftarrow & M \\[0.3em]
            \pi \downarrow          &                             &   \\[0.3em]
            BT                   &                             &
          \end{array}
\end{equation}
where $f$ is an inclusion of a fiber $M$. Then it induces the following sequence
$$ H^*(BT) \stackrel{\pi^*} \rightarrow H^*_T(M) \stackrel{f^*} \rightarrow H^*(M). $$
In particular, $H^*_{T}(M)$ has an $H^*(BT)$-module structure via the map $\pi^*$ such that
$$ x \cdot \alpha = \pi^*(x)\cup \alpha $$ for $x \in H^*(BT)$ and $\alpha \in H^*_{T}(M)$.
\begin{definition}
    Let $(M,\omega)$ be a symplectic manifold. We say that a $T$-action on $(M,\omega)$ is \emph{Hamiltonian}
    if there exists a smooth map $\mu : M \rightarrow \mathfrak{t}^*$ such that
    \[
        d \langle \mu, X \rangle = \omega (X, \cdot)
    \]
    for every $X \in \mathfrak{t}$. We call $\mu$ \emph{a moment map} for the $T$-action.
\end{definition}

\begin{remark}
Note that if $\mu$ is a moment map for a Hamiltonian $T$-action on $(M,\omega)$, then $\mu + c$ is also a moment map for
any $c \in \mathfrak{t}^*$. Thus a moment map is not unique.
\end{remark}

The equivariant cohomology of Hamiltonian $T$-action has a remarkable property as follows.

\begin{theorem}\label{theorem_equivariant_formality}\cite{Ki}
Let $(M,\omega)$ be a closed symplectic manifold equipped with a Hamiltonian $T$-action. Then $M$ is equivariantly formal, that is,
$H^*_T(M)$ is a free $H^*(BT)$-module so that $$H^*_{T}(M) \cong H^*(M) \otimes H^*(BT).$$
Equivalently, the map $f^*$ is surjective with the kernel
$\langle x_1, \cdots, x_m \rangle \cdot H^*_{T}(M)$ where
$\langle x_1, \cdots, x_m \rangle$ is an ideal of $H^*(BT)$ generated by degree two elements $x_1, \cdots, x_m$ and
$\cdot$ denotes the scalar multiplication of the $H^*(BT)$-module structure on $H^*_{T}(M)$.
\end{theorem}

\subsection{Localization theorem}
\label{ssecLocalizationTheorem}
For a given $k \in \Z_{\geq 0}$ and an element $\alpha \in H^k_{T}(M)$, Theorem \ref{theorem_equivariant_formality} implies that $\alpha$ can be uniquely expressed as
\[
     \alpha = \alpha_k \otimes 1 + \sum_{i=1}^m \alpha_{k-2}^i \otimes x_i + \sum_{1 \leq i,j \leq m} \alpha_{k-4}^{i,j} \otimes x_ix_j + \cdots
\]
where $\alpha_i^{J} \in H^i(M)$ for every $i \leq k$ and $J$ is a multiset whose elements are in  $[m] = \{1, \cdots, m\}$.
We denote the set of multisets with elements in $[m]$ by $[m]^{\mathrm{mul}}$.
With this notation, we have $f^*(\alpha) = \alpha_k$.

\begin{definition}
An \textit{integration along the fiber $M$} is an $H^*(BT)$-module homomorphism $\int_M : H^*_{T}(M) \rightarrow H^*(BT)$ defined by
\[
    \int_M \alpha = \langle \alpha_k, [M] \rangle \cdot 1 + \sum_{i=1}^m \langle \alpha_{k-2}^i, [M] \rangle \cdot x_i + \cdots
\]
for every $k \in \Z_{\geq 0}$ and any $\alpha \in H^k_{T}(M)$ where $[M]$ is the fundamental homology class of $M$.
\end{definition}

Note that $\langle \alpha_j^{J}, [M] \rangle = 0$ for any $J \subset [m]^{\mathrm{mul}}$ and $j < \dim M = 2n$. 
Also, $\alpha_j^J = 0$ for every $j > 2n$ for a dimensional reason, and therefore we have  
\[
    \int_M \alpha = \sum_{\substack{J \in [m]^{\mathrm{mul}} \\ |J| + 2n = k}} \langle \alpha_{2n}^J,[M] \rangle x^J
\]
for every $k \in \Z_{\geq 0}$ and any $\alpha \in H^k_{T}(M)$ 
where $x^J = \prod_{j \in J} x_j$. This leads to the following corollary.

\begin{corollary}\label{corollary_localization_degree_2n}
Let $\alpha \in H^*_{S^1}(M)$ such that $\deg \alpha \leq \dim M$. Then we have
$$ \int_M \alpha = \langle f^*(\alpha), [M] \rangle.$$
\end{corollary}

Let $M^{T}$ be the fixed point set and let $F \subset M^{T}$ be a fixed component with an inclusion map $i_F : F \hookrightarrow M$. Then it induces a ring homomorphism $$i_F^* : H^*_{T}(M) \rightarrow H^*_{T}(F) \cong H^*(F) \otimes H^*(BT).$$
For any $\alpha \in H^*_{T}(M)$, the image $i_F^*(\alpha)$ is called 
\textit{the restriction of $\alpha$ to $F$} and is denoted by $\alpha|_F$.
The following theorem due to Atiyah-Bott \cite{AB} and Berline-Vergne \cite{BV} states that the integration 
$\int_M \alpha$ can be calculated in terms of the fixed point data. 

\begin{theorem}\label{theorem_localization}(ABBV-localization)
For any $ \alpha \in H^*_{T}(M)$, we have
$$\int_M \alpha = \sum_{F \subset M^{T}} \int_F \frac{\alpha|_F}{\Lambda_F} $$
where $\Lambda_F$ is the equivariant Euler class of the normal bundle of $F$.
In particular, if every fixed point is isolated, then
\[
\int_M \alpha = \sum_{F \in M^T}
\frac{\alpha|_F}{\Lambda_F}.
\]
\end{theorem}

Recall that the $l$-th Hodge-Riemann bilinear form is given by
\begin{displaymath}
            \begin{array}{cccc}
                \HR_l : & H^l(M)  \times  H^l(M) & \longrightarrow & \R \\[0.5em]
                 & (\alpha, \beta) & \longmapsto &  <\alpha \beta {[\omega]}^{n-l}, [M]> \\[0.5em]
            \end{array}
        \end{displaymath}
for $l=0,1,\cdots,n$. Let $\alpha$ and $\beta$ be any elements in $H^l(M)$.
Since $f^*$ is surjective by Theorem \ref{theorem_equivariant_formality}, 
we can find $\widetilde{\alpha},
\widetilde{\beta}, [\widetilde{\omega}] \in H^*_T (M)$
such that $f^*(\widetilde{\alpha}) = \alpha$,
$f^*(\widetilde{\beta}) = \beta$, and
$f^*([\widetilde{\omega}]) = [\omega]$ and hence we get
\[
\int_M \widetilde{\alpha} \widetilde{\beta}
[\widetilde{\omega}]^{n-l} = \langle \alpha \beta
{[\omega]}^{n-l}, [M] \rangle
\]
by Corollary \ref{corollary_localization_degree_2n}.
Thus we can compute $\langle \alpha \beta [\omega]^{n-l}, [M] \rangle$ by applying the ABBV-localization theorem to
$\widetilde{\alpha} \widetilde{\beta} [\widetilde{\omega}]^{n-l}$.

\subsection{Cartan models}
\label{ssecCartanModels} Note that the choice of a class
$[\widetilde{\omega}] \in H^2_T(M)$ satisfying
$f^*([\widetilde{\omega}]) = [\omega]$ is parametrized by 
a moment map. To
understand $[\widetilde{\omega}]$ in more detail, we
briefly overview the Cartan model of $H^*_T(M)$ as
follows. (See also \cite{GS2}.) Let us consider the set
of \emph{equivariant $q$-forms}
\[
    \Omega^q_T(M) = \bigoplus_{2i + j = q} S^i(\mathfrak{t}^*) \otimes \Omega^j(M)^T
\]
where $S^i(\mathfrak{t}^*)$ denotes the set of degree $i$ elements in the symmetric tensor algebra of $\mathfrak{t}^*$ and $\Omega^j(M)^T$ is the set of
$T$-invariant differential $j$-forms on $M$.
Then we may think of an element $\alpha \in \Omega^*_T(M)$ as a map from $\mathfrak{t}$ to $\Omega^*(M)^T$.
We call $(\Omega^*_T, d_T)$ the \emph{Cartan complex} where the differential is defined by 
\[
    d_T := 1 \otimes d + \sum_{j=1}^m x_i \otimes i_{X_i}, \quad d_T(f \otimes \alpha) = f \otimes d\alpha + \sum_{j=1}^m x_i f \otimes i_{X_i}\alpha
\]
for any $f \otimes \alpha \in S^*(\mathfrak{t}^*) \otimes \Omega^*(M)^T$
where $\{X_1, \cdots, X_m\}$ and $\{x_1, \cdots, x_m\}$ are the basis, which we have chosen in \eqref{equation_polynomial_ring},
of $\mathfrak{t}$ and $\mathfrak{t}^*$, respectively.
Then it is not hard to check that $d_T^2 = 0$ by direct computation.
The \emph{equivariant de Rham theorem} states that
\[
    H^*_T(M) \cong H(\Omega^*_T(M), d_T).
\]
Now, let $\mu = (\mu_1, \cdots, \mu_m) : M \rightarrow \mathfrak{t}^*$ be a moment map where $m = \dim T$.
Since each component of $\mu$ is $T$-invariant, we may regard $\mu$ as an element of $S^1(\mathfrak{t}^*) \otimes \Omega^0(M)^T \subset \Omega^2_T(M)$ such that
\[
    \mu = x_1 \otimes \mu_1 + \cdots + x_m \otimes \mu_m.
\]
Since $\omega$ is also $T$-invariant, $\omega$ can be regarded as the element 
$1 \otimes \omega \in S^0(\frak{t}^*) \otimes \Omega^2(M)^T \subset \Omega_T^2(M)$. Define 
\[
    \widetilde{\omega}_\mu := \omega - \mu = 1 \otimes \omega - \sum_{i=1}^m x_i \otimes \mu_i \in \Omega^2_T(M).
\]
We call $\widetilde{\omega}_\mu$ the \emph{equivariant symplectic form with respect to $\mu$}.
Then 
\[
    \begin{array}{ccl}
        d_T(\widetilde{\omega}_\mu) & = & d_T(\omega - \mu) \\
                          & = & 1 \otimes d\omega - \sum_{j=1}^m x_i \otimes d\mu_i + \sum_{j=1}^m x_i \otimes i_{X_i}\omega  \\
                                          & = & \sum_{j=1}^m x_i \otimes (i_{X_i}\omega - d\mu_i)\\
                                          & = & 0
    \end{array}
\]
so that $\widetilde{\omega}_\mu$ is $d_T$-closed, and therefore $\widetilde{\omega}_\mu$ represents an equivariant 
cohomology class $[\widetilde{\omega}_\mu] \in H^2_T(M)$ which we call the 
\emph{equivariant symplectic class with respect to $\mu$}. 
Then we immediately obtain the following corollary from the definition of $\widetilde{\omega}_\mu$.
\begin{lemma}\label{lemma_equivariant_symplectic}
    Let $v \in M^T$ be an isolated fixed point. Then
    \[
        [\widetilde{\omega}_\mu]|_v = -\sum_{j=1}^m x_i \otimes \mu_i(v) = -\mu(v) \in \mathfrak{t}^* = S^1(\mathfrak{t}^*) \cong H^2(BT).
    \]
\end{lemma}

\section{The Graph cohomology of Hamiltonian GKM manifolds}
\label{secGraphCohomologyAndEquivariantCohomologyOfHamiltonianGKMManifolds}

In this section, we briefly review the theory of Hamiltonian GKM-manifolds and GKM graphs, following \cite{GKM} and \cite{GZ}.

\subsection{GKM manifolds}
\label{ssecGKMManifolds}
Let $(M,\omega)$ be a $2n$-dimensional closed symplectic manifold
and let $T$ be an $m$-dimensional torus with its Lie algebra $\mathfrak{t}$ for some integer $m \geq 2$.
Suppose that $T$ acts on $(M,\omega)$ in a Hamiltonian fashion
with a moment map $\mu : M \longrightarrow \mathfrak{t}^*$.

\begin{definition} \label{definition: GKM manifold}
The triple $(M, \omega, \mu)$
is called a {\em Hamiltonian GKM manifold} if
\begin{enumerate}
\item the fixed point set $M^T$ is finite, and
\item for each $v \in M^T,$ the weights
$\alpha_{j, v} \in \mathfrak{t}^*,$ $j=1, \cdots, n,$
of the one-dimensional isotropy $T$-representations on $T_v M$ are
pairwise linearly independent.
\end{enumerate}
\end{definition}

A Hamiltonian GKM manifold $(M, \omega, \mu)$ defines
a graph $\Gamma := \Gamma(M,\omega,\mu)$, called a {\em GKM graph}, 
where the vertex set and the oriented edge set are defined as follows:
\begin{itemize}
\item the vertex set $V_\Gamma$ is equal to $M^T,$
\item the oriented edge set $E_\Gamma$ consists of pairs
$(p,q) \in V_\Gamma \times V_\Gamma$ ($p \neq q$)
such that $p$ and $q$ are in the same component of the $H$-fixed point set $M^H$
for some codimension one subtorus $H$ of $T$.
Equivalently, $(p,q) \in E_\Gamma$ if and only if $p$ and $q$ are
contained in a $T$-invariant two-sphere in $M.$
In particular, $(p,q) \in E_\Gamma$ if and only if $(q,p) \in E_\Gamma$.
\end{itemize}
We call the two conditions (1) and (2) in Definition \ref{definition: GKM manifold} the {\em GKM conditions}.
Note that each $v \in V_\Gamma$
is contained in exactly $n$ edges, i.e., $\Gamma$ is an
$n$-valent graph. Indeed, for each fixed point $v \in M^T$, the tangential $T$-representation on $T_vM$ splits into the sum of one-dimensional irreducible representations so that 
\[
T_vM = \oplus_{j=1}^n \xi_j
\]
where $\xi_j$ is a one-dimensional irreducible $T$-representation with weight $\alpha_{j,v} \in \frak{t}^*$
for $j=1,\cdots, n$. Then any element $z_j \in
\xi_j \subset T_vM$ is fixed by the adjoint action of
$\ker \alpha_{j,v}$, and therefore $\xi_j$ is fixed by the codimension one subtorus 
$H_j := \exp (\ker \alpha_{j,v})$ of $T$. 
By the second GKM condition (2), the connected component of $M^{H_j}$ containing $v$ is of dimension two, i.e., it is a two-sphere
and it contains exactly two fixed points, say $v$ and $v_j'$, of the $T$-action. 
Thus there exist $n$ fixed points $v_1', \cdots, v_n'$ of the $T$-action such that $(v, v_j') \in E_\Gamma$ for each $j=1,\cdots, n$. 
Furthermore, the GKM condition (2) implies that there is no more edge 
containing $v$ except for $(v, v_j')$'s for $j=1,\cdots,n$.


For an oriented edge $e=(p,q) \in E_\Gamma$, we
denote by $i(e)$ and $t(e)$ the initial vertex
$p$ and the terminal vertex $q$ of $e,$
respectively. For each $\xi \in \mathfrak{t}$,
let $\mu_\xi := \langle \mu, \xi \rangle$ where
$\langle~, \rangle$ is the canonical pairing of
$\mathfrak{t}^*$ and $\mathfrak{t}.$ We say $\xi$
is {\em generic} if
\[
\mu_\xi \big( i(e) \big) ~ \ne ~ \mu_\xi \big( t(e)
\big)
\]
for every $e \in E_\Gamma$.
In other words, $\xi$ is generic if $\xi$ is not perpendicular to $\mu(q) - \mu(p)$ for any edge $(p,q)$ of $\Gamma$.

Now, fix a generic $\xi \in \mathfrak{t}$. We say that
$e \in E_\Gamma$ is {\em ascending} (resp. {\em descending}) with respect to $\xi$
if
$
\mu_\xi \big( i(e) \big) < \mu_\xi \big( t(e) \big)
$
(resp. $\mu_\xi \big( i(e) \big)
> \mu_\xi \big( t(e) \big)$). 
The {\em index} of $v \in
V_{\Gamma}$, denoted by $\lambda_v$, is defined as twice the number of
descending edges starting at $v.$

\begin{remark} \label{remark_Ham_circle_action_property}
We can always take a generic element $\xi$ lying on the lattice of $\mathfrak{t}$ so that
$\xi$ generates a circle subgroup $S^1$ of $T$.
Then the Hamiltonian $S^1$-action generated by $\xi$ has a moment map $\mu_{\xi} =
\langle \xi, \mu \rangle$ and the genericity of $\xi$ implies that
the fixed point set $M^{S^1}$ for the $S^1$-action
is the same as $M^T$. Moreover, $\mu_{\xi}$ is a
Morse function on $M$ such that each fixed point $v \in
M^{S^1}$ has a Morse index equal to $\lambda_v.$ See \cite{Au} for more details.
\end{remark}

\begin{definition}\label{definition_index_increasing}
Let $\xi \in \frak{t}$ be a generic vector. 
$\Gamma$ is called {\em index increasing with
respect to $\xi \in \mathfrak{t}$} if
\[
\mu_\xi \big(i(e) \big) < \mu_\xi \big(t(e) \big) 
\quad \text{implies} \quad \quad   ~ ~ \lambda_{i(e)} < \lambda_{t(e)}
\]
for every $e \in E_\Gamma$.
If $\Gamma$ is index increasing with respect to some
$\xi \in \mathfrak{t},$ then $\Gamma$ is simply called
{\em index increasing}.
\end{definition}
\begin{remark}\label{remark_ind_inc_rev_orientation}
We note that if $\Gamma$ is index
increasing with respect to $\xi \in \mathfrak{t},$
then $\Gamma$ is also index increasing with respect to
$-\xi.$
\end{remark}


\subsection{Graph cohomology rings}
For each $e \in E_\Gamma$, we denote by
$S_e^2$ the unique $T$-invariant two-sphere containing $i(e)$
and $t(e)$. 
Let us define a function $\alpha$, called an {\em axial function of $\Gamma$}, which assigns the weight of the one-dimensional tangential $T$-representation on $T_{i(e)} \, S_e^2$ 
for each $e \in E_\Gamma$ :
\begin{flalign*}
& & \alpha : E_\Gamma \longrightarrow \mathfrak{t}^*,
\quad e \longmapsto \alpha(e). & &
\end{flalign*}
\begin{notation}
For the sake of simplicity, we denote by $(p,q)$ the oriented edge $e$ such that $i(e) = p$ and $t(e) = q$. Also, we denote
$\alpha \big((p,q) \big)$ by $\alpha(p,q).$
\end{notation}

\begin{definition} \label{definition: graph
cohomology}
For a given pair $(\Gamma, \alpha)$, the {\em graph cohomology ring} $H(\Gamma, \alpha)$
is defined by
\[
\{ h: V_\Gamma \rightarrow \sym(\mathfrak{t}^*) ~|~ h
\big( t(e) \big) - h \big( i(e) \big) \equiv 0 \mod
\alpha(e) \text{ for every } e \in E_\Gamma \},
\]
where $\sym(\mathfrak{t}^*)$ is identified with a
polynomial ring $\R[x_1,\cdots,x_m]$ as in (\ref{equation_polynomial_ring}).
\end{definition}
The product structure on $H(\Gamma,\alpha)$ is defined by 
\[
    (h_1 \cdot h_2) (v) := h_1(v)h_2(v) \in S(\mathfrak{t}^*) \cong \R[x_1,\cdots,x_m]
\]
for every $h_1, h_2 \in H(\Gamma,\alpha)$.
The graph cohomology ring $H(\Gamma,\alpha)$ has a natural $\Z$-grading given by
\[
H^i (\Gamma, \alpha) := H (\Gamma, \alpha)
\cap \Map (V_\Gamma,
\sym^i(\mathfrak{t}^*))
\]
where $\sym^i(\mathfrak{t}^*)$ is the $\R$-subspace of $\sym(\mathfrak{t}^*)$ generated by $i$-times symmetric tensor products of elements in $\mathfrak{t}^*$
for $i > 0$. When $i = 0$, we put $\sym^0(\mathfrak{t}^*) = \R$.

Together with the product structure, $H(\Gamma,\alpha)$ becomes a commutative $\Z$-graded ring.
Also, any $\sym(\mathfrak{t}^*)$-valued constant function on $V_\Gamma$ is an element of $H(\Gamma, \alpha)$
and hence $\sym(\mathfrak{t}^*)$ is a subring of $H(\Gamma, \alpha)$.
Therefore, $H(\Gamma, \alpha)$ is an $\sym(\mathfrak{t}^*)$-algebra with the unit $1\in \sym^0(\mathfrak{t}^*) = \R$.
\begin{lemma} \label{lemma_zero_at_a_vertex}
Let $\Gamma$ and $\alpha$ be given as above.
\begin{enumerate}
\item Let $h \in H^1 (\Gamma, \alpha)$
and $e \in E_\Gamma$. If $h \big( t(e) \big) = 0$, then $h \big( i(e) \big) = k \cdot \alpha (e)$ for some $k \in \R.$
\item Let $h \in H^i (\Gamma, \alpha)$ for
some $i \le n-1.$ If $h(v)=0$ for every vertex $v$
except one, then $h=0.$
\end{enumerate}
\end{lemma}

\begin{proof}
(1) is straightforward by definition of
$H(\Gamma,\alpha)$.
For (2), assume that $h(v_0) \ne 0$ for some $v_0 \in
V_\Gamma$ and $h(v)=0$ for any other vertex $v \neq v_0$. Then $\alpha(v, v_0)$
divides $h(v_0)$ for every $v$ adjacent to
$v_0$. Also, these $\alpha(v, v_0)$'s are pairwise
linearly independent by the GKM condition (2). Thus $h(v_0)$ should be of
polynomial degree at least $n$ since $\sym (\mathfrak{t}^*)$ is a UFD.
This contradicts that $\deg h(v_0) \le n-1$, and therefore $h(v_0) = 0$. 
\end{proof}

\subsection{Equivariant Thom classes}
Let $\xi \in \mathfrak{t}$ be a generic vector.
A {\em path} of $\Gamma$ is a
sequence of vertices $(v_0, \cdots, v_l)$ of $\Gamma$ such that $(v_j, v_{j+1}) \in E_\Gamma$ for every $j = 0, \cdots, l-1.$
We say that a path $(v_0, \cdots, v_l)$ is {\em ascending} (resp. {\em descending}) with respect to $\xi$ if each 
$(v_j, v_{j+1})$ is ascending (resp. descending) with respect to $\xi$ for every $j$. 

For each $v \in V_\Gamma,$ let $E_v^\uparrow$ (resp. $E_v^\downarrow$)
be the set of ascending (resp. descending) edges with respect to $\xi$ having the initial vertex $v$.
Note that
$|E_v^\downarrow| = \lambda_v/2$ where $\lambda_v$ is the index of $v$ (with respect to $\xi$) and $\lambda_v$ 
is equal to the Morse index of $v$ with respect to $\mu_{\xi},$ see Remark \ref{remark_Ham_circle_action_property}.

For each $h \in H(\Gamma,\alpha)$, define a {\em support} of $h$ by
\[
\supp h := \{v \in V_\Gamma ~|~ h(v) \neq 0 \}.
\]
Guillemin-Zara \cite{GZ} proved that there exists a nice basis of $H(\Gamma, \alpha)$ as an $\sym
(\mathfrak{t}^*)$-module whose elements are called {\em equivariant Thom classes}.

\begin{theorem}\cite[Theorem 1.5, 1.6]{GZ} \label{theorem_Thom_class}
Let $(M,\omega,\mu)$ be a Hamiltonian GKM manifold
with its GKM graph $(\Gamma, V_\Gamma, E_\Gamma).$ If
$\Gamma$ is index increasing with respect to some generic $\xi
\in \mathfrak{t},$ then for
each $v \in V_\Gamma$, there exists a unique element
$\tau_v^+$ of $H^{\lambda_v/2} (\Gamma, \alpha)$ satisfying
\begin{enumerate}
\item $\supp \tau_v^+ \subset \big\{~ v^\prime \in V_\Gamma ~|~
\text{there exists an ascending path with respect to}~\xi~\text{from} ~v ~\text{to}~v^\prime ~ \big\},$ and 
\item $\tau_v^+ (v) = \Lambda_v^+ := \prod_{e \in E_v^\downarrow} ~
\alpha (e).$
\end{enumerate}
Furthermore, the set $\{ \tau_v^+ \}_{v \in V_\Gamma}$ forms a basis of $H(\Gamma, \alpha)$
as an $\sym (\mathfrak{t}^*)$-module.
\end{theorem}

We call $\tau_v^+$ the {\em equivariant Thom class
for $v \in V_\Gamma$ with respect to $\xi$}. 
As in Remark \ref{remark_ind_inc_rev_orientation}, if $\Gamma$ is index
increasing with respect to $\xi,$ then $\Gamma$ is also index increasing with respect to $-\xi.$
We denote by $\tau_v^-$ the equivariant Thom class for $v \in V_\Gamma$ with respect to $-\xi.$
Then by Theorem \ref{theorem_Thom_class}, we have
\[
\tau_v^- (v)= \Lambda_v^- := \prod_{e \in E_v^\uparrow} ~
\alpha (e) \qquad \text{and} \qquad \Lambda_v = \Lambda_v^+
\cdot \Lambda_v^-,
\]
where $\Lambda_v$ is the equivariant Euler class of the
normal bundle to $v$ in $M$.

\subsection{GKM description of equivariant cohomology}
Recall that the inclusion map $\imath : M^T \hookrightarrow M$
induces an $H^* (BT)$-algebra homomorphism
\[
\imath^* : H^*_T (M) \rightarrow H^*_T (M^T)
\cong \bigoplus_{v \in M^T} H^*_T (\{v\}). 
\]
In particular, for each fixed point $v \in M^T$, the inclusion $\imath_v : \{v\} \hookrightarrow M$ induces the map
\[
\imath_v^* : H^*_T (M) \rightarrow H^*_T
(\{v\}) \cong H^*(\{v\})\otimes H^*(BT)
\cong H^*(BT), 
\]
and the image $\imath_v^* (\beta)$  is called the {\em restriction of $\beta$ to $v$} and is denoted by $\beta|_v$ 
for every $\beta \in H^*_T (M)$. See Section \ref{ssecLocalizationTheorem}.
The following theorem is a symplectic
version of the theorem \cite{GKM} due to Goresky,
Kottwitz, and MacPherson, which  enables us to identify $H^*_T(M)$ with $H(\Gamma,\alpha)$.

\begin{theorem}\cite{GKM}\label{theorem_GKM}
Let $(M, \omega, \mu)$ be a closed Hamiltonian GKM manifold
with its GKM graph $(\Gamma, V_{\Gamma}, E_{\Gamma})$. Then the map
\[
H^*_T(M) \longrightarrow H(\Gamma, \alpha), \quad
\beta \longmapsto h_\beta
\]
is an $\sym(\mathfrak{t}^*)$-algebra isomorphism
where
\[
h_\beta(v):=\beta|_v
\]
for each $v \in V_\Gamma$.
The
image of $H_T^{2l} (M)$ under this isomorphism is
$H^l(\Gamma, \alpha)$ for every integer $l \geq 0.$
\end{theorem}

%

Let $e \in E_\Gamma$ be any oriented edge. Let us label the $n$ edges outward from $i(e)$ by 
$e_{1,i(e)}, \cdots, e_{n,i(e)}$. Also, let
\[
\alpha_{j, i(e)} := \alpha(i(e_{j,i(e)}), t(e_{j,i(e)})) = \alpha(i(e), t(e_{j,i(e)})).
\]
Also, we can define $\alpha_{j,t(e)}$'s in a similar way. 

\begin{lemma}\cite[Proposition 2.2]{GZ} \label{lemma_GKM_weight_pair}
For each oriented edge $e \in E_\Gamma,$ we can
reorder $e_{j,i(e)}$'s and $e_{j,t(e)}$'s so that
\begin{equation}
\label{equation: weight reorder} \alpha_{n, t(e)} =
-\alpha_{n, i(e)} = - \alpha(e) \qquad \text{and}
\qquad \alpha_{j, t(e)} \equiv \alpha_{j, i(e)} \mod
\alpha(e)
\end{equation}
for each $1 \le j \le n-1.$
\end{lemma}

\section{Hodge-Riemann bilinear forms}
\label{secHodgeRiemannBilinearFormInHigherDimension}

Let $(M,\omega)$ be a $2n$-dimensional closed symplectic manifold and let $T$ be an $m$-dimensional ($m \ge
2$) compact torus acting on $(M,\omega)$ in a Hamiltonian fashion with a moment map $\mu : M
\rightarrow \mathfrak{t}^*.$ Assume the action is GKM and 
the corresponding GKM graph $\Gamma$ is index increasing with respect to some generic vector $\xi \in \mathfrak{t}^*$
so that the equivariant Thom class exists for every vertex $v \in V_\Gamma$ by Theorem \ref{theorem_Thom_class}.  
In the section, we compute the matrix $A_l(M,\omega)$ presenting the Hodge-Riemann bilinear form $\mathrm{HR}_l$ for each $l = 0, \cdots, n.$

For a fixed $l$ with $0 \leq l \leq n$, let $b_l := b_l(M)$ be the $l$-th Betti number of $M$ and let
\[
\{p_1, \cdots, p_{b_l}\} \qquad \text{and} \qquad
\{q_1, \cdots, q_{b_l}\}
\]
be the set of vertices of index $l$ and index $2n-l$, respectively. Then Theorem \ref{theorem_Thom_class} and 
Theorem \ref{theorem_equivariant_formality} imply that each of
\[
    \mathcal{B}_{l}^+ := \{f^* \tau_{p_1}^+, \cdots, f^*\tau_{p_{b_l}} \} \quad \text{and} \quad \mathcal{B}_{l}^- := \{ f^*\tau_{q_1}^-, \cdots, f^*\tau_{q_{b_l}}^- \}
\]
forms a basis of $H^{l}(M)$ where $f : M
\hookrightarrow M \times_T ET$ is an inclusion of a
fiber $M,$ see (\ref{diagram : M-bundle over BT}).
Then the Hodge-Riemann form $\mathrm{HR}_l$ is
represented by the following $b_l \times b_l$ matrix
\begin{equation}\label{equation_matrix_A}
 A_l (M,\omega) = (a_{jk})_{1 \le j, \,
k \le b_l} := \Big( \HR_l(f^*\tau_{p_k}^+, f^*
\tau_{q_j}^-) \Big)_{1 \le j, \, k \le b_l}.  
\end{equation}
It is straightforward that $(M,\omega)$ satisfies the hard Lefschetz property if and only if $A_l (M,\omega)$ is non-singular for every $l=0,1,\cdots, n$. 

\subsection{$\Theta$ function and $\vol$ function}
\label{ssecVolAndTheta}
To compute each entry $a_{jk}$ of $A_l(M,\omega)$, we define two functions $\Theta$ and $\vol$ as follows. 
The function $\vol$, called the {\em volume function}, is defined by 

\begin{equation*}
\vol : E_\Gamma \longrightarrow \R, \quad (p,
q) \longmapsto \Big( \mu(q)-\mu(p) \Big)
\Big/ \alpha (p, q)
\end{equation*}
for every $(p,q) \in E_\Gamma$. 

\begin{lemma} \label{lemma_symplectic_volume}
For any $(p, q) \in E_\Gamma$, the symplectic area of the $T$-invariant two-sphere $S_{(p, q)}^2$ containing $p$ and $q$ is equal to $\vol(p, q)$. 
In particular, $\vol(p, q)$ is a positive real number.
\end{lemma}

\begin{proof}
Let $i : S_{(p,q)}^2 \hookrightarrow M$ be the
embedding of $S_{(p,q)}^2$ into $M.$ 
Then the symplectic volume $\int_{S_{(p,q)}^2} i^* \omega$ is equal to the integration along the fiber
$\int_{S_{(p,q)}^2} i^* [\widetilde{\omega}_\mu]$ where $[\widetilde{\omega}_\mu] \in H_T^2 (M)$ 
is the equivariant symplectic class with respect to $\mu$.
By the ABBV-localization theorem (Theorem \ref{theorem_localization})
and Lemma \ref{lemma_equivariant_symplectic}, 
we have
\[
\int_{S_{(p,q)}^2} i^* [\widetilde{\omega}_\mu] ~=~
\frac{[\widetilde{\omega}_\mu]|_p}{e_p} +
\frac{[\widetilde{\omega}_\mu]|_q}{e_q} ~=~ \frac {-\mu(p)}
{\alpha(p,q)}+\frac {-\mu(q)} {\alpha(q,p)} ~=~ \frac
{\mu(q) - \mu(p)} {\alpha(p,q)}.
\]
This completes the proof.
\end{proof}

Now, following \cite[p. 453]{GT}, we define the function $\Theta$ by
\[
    \begin{array}{ccccl}
        \Theta & : & E_\Gamma & \rightarrow & Q(\mathfrak{t}^*) \\
                    &    &    (p,q) & \mapsto & \Theta(p,q) := \displaystyle \frac{\rho_{\alpha(p,q)} (\Lambda^+_p)}{\rho_{\alpha(p,q)} (\Lambda^+_{q} / \alpha(q,p))} \\
    \end{array}
\]
where $Q (\mathfrak{t}^*)$ is the quotient field of
$\sym(\mathfrak{t}^*)$ and $\rho_{\alpha(p,q)}$ is the
canonical extension of the projection map 
\[
	X \mapsto X - \frac{\langle X, \xi \rangle }{\langle \alpha(p,q), \xi \rangle }\alpha(p,q) \quad \text{on} ~\mathfrak{t}^*
\]
to $\sym(\mathfrak{t}^*)$. Note that $\Theta(p,q) \in Q(\frak{t}^*)$ is a nonzero element (rational function) in $Q(\frak{t}^*)$ by the GKM conditions. 
Moreover, it has an integer value when $\lambda_q - \lambda_p = 2$. See \cite[Theorem 2.4]{ST}.


\begin{figure}[H]
\begin{center}

\begin{pspicture}(-6,-3)(6,3)\footnotesize

\psline[arrowsize=5pt]{->}(1,-1.5)(1,2)
\psline(0.75,0)(5.5,0)

\psline(1,0)(3.25,1.5)
\psline[arrowsize=5pt]{->}(1,0)(1.25,-1)
\psline[linestyle=dotted](1.25,-1)(4.25,1)
\psline[linestyle=dotted](3,-1)(5.25,0.5)

\psdots[dotsize=4pt](2.75,0)(4.5,0)(1,0)(3.25,1.5)

\psline[arrowsize=6pt,arrowlength=0.6,arrowinset=0.7,
linestyle=none]{<}(2.5,1)(3.25,1.5)
\psline[arrowsize=6pt,arrowlength=0.6,arrowinset=0.7,
linestyle=none]{<}(3.5,0.5)(4.25,1)
\psline[arrowsize=6pt,arrowlength=0.6,arrowinset=0.7,
linestyle=none]{<}(3.75,-0.5)(5.25,0.5)

\psline[arrowsize=5pt]{->}(3.25,1.5)(5.25,0.5)
\psline[arrowsize=5pt]{->}(1,0)(3,-1)

\psline[arrowsize=6pt,arrowlength=0.6,arrowinset=0.7,
linestyle=none]{<}(2.4,-0.7)(3,-1)
\psline[arrowsize=6pt,arrowlength=0.6,arrowinset=0.7,
linestyle=none]{<}(2.5,-0.75)(3,-1)

\psline[arrowsize=6pt,arrowlength=0.6,arrowinset=0.7,
linestyle=none]{<}(4.65,0.8)(5.25,0.5)
\psline[arrowsize=6pt,arrowlength=0.6,arrowinset=0.7,
linestyle=none]{<}(4.75,0.75)(5.25,0.5)

\pscurve[linewidth=0.2pt]{->}(2.75,0)(1.5,-1.5)(0.25,-2)
\pscurve[linewidth=0.2pt]{->}(4.5,0)(4.9,-1)(5,-2)
\pscurve[linewidth=0.2pt](3.25,1.5)(4.05,1.45)(4.25,1)
\pscurve[linewidth=0.2pt](3.25,1.5)(5.05,1.7)(5.25,0.5)

\uput[dl](1, 0){$\mu(p)$} \uput[u](3.25,1.5){$\mu(q)$}
\uput[l](1,2){$\xi$}
\uput[d](0.25,-2){$\rho_{\alpha(p,q)}(\Lambda_p^+)$}
\uput[d](5,-2){$\rho_{\alpha(p,q)}(\Lambda_q^+/\alpha(q,p))$}
\uput[ur](4.05,1.25){$a$} \uput[ur](5.05,1.6){$b$}
\uput[r](5.7,1.2){$\Theta(p,q)=\frac a b$}

\psdots[dotsize=4pt](-6,0)(-3.75,1.5)

\psline(-6,0)(-3.75,1.5)
\psline[arrowsize=5pt]{->}(-6,0)(-5.75,-1)
\psline[arrowsize=5pt]{->}(-3.75,1.5)(-1.75,0.5)
\psline[arrowsize=5pt]{->}(-7,-0.2)(-7,1.3)
\psline[arrowsize=5pt]{->}(-6,0)(-6,1)
\psline[arrowsize=5pt]{->}(-3.75,1.5)(-4.25,2.5)

\uput[dl](-6, 0){$\mu(p)$} \uput[ur](-3.75,1.5){$\mu(q)$}
\uput[l](-7,1.3){$\xi$}
\uput[d](-5.75,-1){$\Lambda_p^+$}
\uput[u](-5.9,1){$\Lambda_p^- / \alpha(p,q)$}
\uput[u](-4.25,2.5){$\Lambda_q^-$}
\uput[d](-1.75,0.5){$\Lambda_q^+/\alpha(q,p)$}
\end{pspicture}

\end{center}
\caption{\label{figure_GT_def} Goldin-Tolman's $\Theta$ function}
\end{figure}
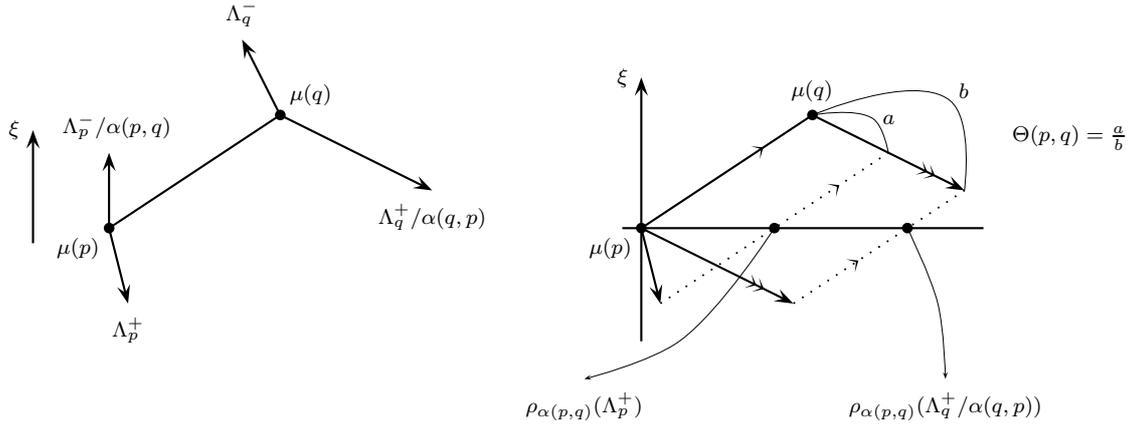

\begin{lemma}\label{lemma_rho_commute}
$\rho_{\alpha(p,q)} = \rho_{\alpha(q,p)}$ for every $(p,q) \in E_\Gamma$.
\end{lemma}

\begin{proof}
	It is straightforward by definition of $\rho$.
\end{proof}

On the other hand, let us consider $\Gamma$ with an opposite generic vector $-\xi \in \mathfrak{t}^*$. 
Then $\Gamma$ is also index-increasing with respect to $-\xi$ and the index of $v$, denoted by $\overline{\lambda}_v$,  with respect to $\mu_{-\xi}$ is given by $\overline{\lambda}_v = 2n - \lambda_v$ for every $v \in V_\Gamma$.
Let $\overline{\Theta}$ be the Goldin-Tolman's $\Theta$ function with respect to $-\xi$ so that
    \[
        \overline{\Theta}(q,p) = \displaystyle \frac{\rho_{\alpha(q,p)} (\Lambda^-_{q})}{\rho_{\alpha(q,p)} (\Lambda^-_{p} / \alpha(p,q))}.
    \]
    By applying Lemma \ref{lemma_rho_commute}, we have
    \begin{equation}\label{equation_Theta_Thetabar}
        \begin{array}{ccl}
        \displaystyle \frac{\Theta (p,q)}{\overline{\Theta} (q,p)} & = & \displaystyle \frac{\rho_{\alpha(q,p)} (\Lambda^-_{p} / \alpha(p,q))}{\rho_{\alpha(q,p)} (\Lambda^-_{q})} \cdot
        \frac{\rho_{\alpha(p,q)} (\Lambda^+_p)}{\rho_{\alpha(p,q)} (\Lambda^+_{q} / \alpha(q,p))} \\[1.3em]
        & = & \displaystyle \frac{\rho_{\alpha(p,q)}(\Lambda_p / \alpha(p,q))}{\rho_{\alpha(p,q)}(\Lambda_{q} / \alpha(q,p))}. \\
        \end{array}
    \end{equation}

\begin{lemma}\label{lemma_super_perfect}
    For any $(p,q) \in E_\Gamma$, we have
    \[
        \displaystyle \frac{\rho_{\alpha(p,q)}(\Lambda_p / \alpha(p,q))}{\rho_{\alpha(p,q)}(\Lambda_{q} / \alpha(q,p))} = 1, 
    \]
    and therefore $\Theta(p,q) = \overline{\Theta}(q,p)$.
\end{lemma}

\begin{proof}
Without loss of generality, we may assume that $(p,q)$ is ascending with respect to $\xi$. By Lemma
\ref{lemma_GKM_weight_pair}, we can give orders on the set of edges
$\{ e_1, \cdots, e_n \}$ having initial vertex $p$ and on $\{e_1',
\cdots, e_n'\}$ having initial vertex $q$ such that
\begin{itemize}
\item $\alpha(e_n) = - \alpha(e_n^\prime) = \alpha(p,q)$, and
\item $\alpha(e_j) = \alpha(e_j^\prime) + c_j \alpha(e_n)
= \alpha(e_j^\prime) + c_j \alpha(p,q)$ for some $c_j \in
\R$
\end{itemize}
for every $j=1, \cdots, n-1$. Then $\alpha(e_1) \cdots
\alpha(e_{n-1}) = \alpha(e_1^\prime) \cdots
\alpha(e_{n-1}^\prime)$ modulo $\alpha(p,q)$ in
$\sym(\mathfrak{t}^*)$. Since $\alpha(p,q)$ is in the kernel of $\rho_{\alpha(p,q)}$, we have
$\rho_{\alpha(p,q)}( \alpha(e_1) \cdots
\alpha(e_{n-1})) = \rho_{\alpha(p,q)}(
\alpha(e_1^\prime) \cdots \alpha(e_{n-1}^\prime))$. Furthermore, the GKM conditions imply that
$\rho_{\alpha(p,q)}(\alpha(e_j)) \neq 0$ and
$\rho_{\alpha(p,q)}(\alpha(e_j^\prime)) \neq 0$ for
every $j=1,\cdots, n-1$. Therefore, 
\[
\begin{array}{ccl}
\rho_{\alpha(p,q)}(\Lambda_p / \alpha(p,q)) & = &
\rho_{\alpha(p,q)}(\alpha(e_1) \cdots
\alpha(e_{n-1})) = \rho_{\alpha(p,q)}(\alpha(e_1^\prime) \cdots \alpha(e_{n-1}^\prime))\\
                                            & =  & \rho_{\alpha(p,q)}(\Lambda_{q} / \alpha(q,p)) \neq 0.\\
\end{array}
\]
This completes the proof.
\end{proof}

\subsection{Computation of $A_l(M,\omega)$}
\label{ssecCoefficientsOfAlMOmega}

Let $\mathbf{v} = (v_0, v_1, \cdots, v_s)$ be an ascending path of $\Gamma$ with respect to a generic $\xi$.
Following \cite{GT}, the {\em length} of $\mathbf{v}$ is defined to be $s$ and denoted by $|\mathbf{v}|$.
For any two vertices $p$ and $q$ in
$V_\Gamma$, let $\sum_p^{q}$ be the set of
ascending paths from $p$ to $q$ :
\[
\Big\{ \, \mathbf{v} = (v_0, v_1, \cdots,
v_{|\mathbf{v}|}) \, \Big| \, v_0= p, \,
v_{|\mathbf{v}|}=q, ~ \lambda_{v_{j+1}} -
\lambda_{v_j} = 2, ~\mathrm{and}~ (v_j, v_{j+1}) \in
E_\Gamma \text{ for any } 0 \le j \le |\mathbf{v}|-1
\, \Big\}
\]
where $|\mathbf{v}| =(\lambda_{q}-\lambda_p)/2.$ Also,
we denote by $\sum_p^{q}(r)$ the subset of $\sum_p^{q}$ consisting of paths passing
through $r$ for $r \in V_\Gamma$.

Now, for $1 \le l \le n$, let $A_l(M,\omega)$ be the $(b_l \times b_l)$-matrix with respect to the bases 
$\mathcal{B}_{l}^+$ and $\mathcal{B}_{l}^-$ defined in \eqref{equation_matrix_A}.
The following proposition states that each entry of $A_l(M,\omega)$ can be computed
by using $\vol$, $\Theta$, and $\mu$.

\begin{proposition} \label{proposition_coeff_higher_dim}
The $(j,k)$-th entry $a_{jk}$ of $A_l (M, \omega)$ is equal to
\begin{equation} \label{equation: coefficient higher dim}
\quad \sum_{r \in V_\Gamma} \quad \Bigg[ ~
\prod_{i=1}^{n-l} \, \big[ \, \mu(r) - d_i \, \big] ~
\Bigg] \cdot \Bigg[ ~ \sum_{\mathbf{v} \in
\sum_{p_k}^{q_j}(r)} ~ \frac{\prod_{i=1}^{n-l} \,
\big[ \, \vol(v_{i-1},v_i) \cdot \Theta (v_{i-1}, v_i)
\, \big]} {\prod_{i \in \{ 0, 1, \cdots, n-l\}
\setminus \{ c_r \}} \, \big[ \, \mu(r)-\mu(v_i) \,
\big]} ~ \Bigg]
\end{equation}
for $1 \le j, \, k \le b_l$ where $d_1, \cdots, d_{n-l}$ are any elements in $\frak{t}^*$ and
$c_r = (\lambda_r-\lambda_{p_k})/2$.
\end{proposition}

To prove Proposition \ref{proposition_coeff_higher_dim}, we use the following theorem due to \cite{GT}. (More general formula can be found in \cite{ST}.)

\begin{theorem}\cite[Theorem 1.6]{GT} \label{theorem_gt}
For any $p, q \in V_\Gamma$, the
following holds:
\[
\tau_p^+ (q) = \Lambda_{q}^+ \cdot
\sum_{\mathbf{v} \in \sum_p^{q}}
\prod_{i=1}^{|\mathbf{v}|}
\frac{\mu(v_i)-\mu(v_{i-1})}
{\mu(q)-\mu(v_{i-1})} \cdot \frac{\Theta
(v_{i-1}, v_i)} {\alpha (v_i, v_{i-1})},
\]
where $\mathbf{v}=(p = v_0, v_1,
\cdots, v_{|\mathbf{v}|}=q).$
\end{theorem}

\begin{remark} \label{remark: reverse notation}
In \cite{GT}, they used the opposite sign convention to ours.
For example, our $\alpha(p,q)$ should be $\alpha(q,p)$ in \cite{GT} and our $\Lambda_p^+$ should be $\Lambda_p^-$ in \cite{GT}. 
Note that the notation $\eta(p,q)$ used in \cite{GT} is the same as $\alpha(q,p)$. Also, $\alpha_p$ (the canonical class) in \cite{GT} 
is the same as $\tau_p^+$ in our paper.
\end{remark}

\begin{proof}[Proof of Proposition \ref{proposition_coeff_higher_dim}]
Let us fix $k$ and $j$ with $1 \leq k, j \leq b_{l}$. By Theorem \ref{theorem_gt}, we have
\[
\tau_{p_k}^+ (r) = \Lambda_r^+ \cdot \sum_{\mathbf{v}
\in \sum_{p_k}^r} \prod_{i=1}^{|\mathbf{v}|} \,
\frac{\mu(v_i)-\mu(v_{i-1})} {\mu(r)-\mu(v_{i-1})}
\cdot \frac{\Theta (v_{i-1}, v_i)} {\alpha (v_i,
v_{i-1})}
\]
for every vertex $r \in V_\Gamma$. Note that the length of $\mathbf{v} \in \Sigma_{p_k}^{r}$ is $\frac{\lambda_r - \lambda_{p_k}}{2}$, which we denote by
$c_r$.
By substituting
\[
\vol (v_{i-1}, v_i) = \displaystyle \frac{\mu(v_i)-\mu(v_{i-1})}{\alpha(v_{i-1}, v_i)}
\]
to the the above formula, we have
\[
\tau_{p_k}^+ (r) = \Lambda_r^+ \cdot \sum_{\mathbf{v} \in
\sum_{p_k}^r} \prod_{i=1}^{|\mathbf{v}|} \,
\frac{\vol(v_{i-1}, v_i) \cdot \Theta (v_{i-1}, v_i)}
{-\mu(r)+\mu(v_{i-1})}.
\]
Similarly, with respect to $-\xi \in \mathfrak{t}$, we have
\[
    \begin{array}{ccl}
        \tau_{q_j}^- (r) & = & \Lambda_r^- \cdot \displaystyle \sum_{\mathbf{v} \in
\sum_r^{q_j}} \, \prod_{i=1}^{|\mathbf{v}|}
\frac{\vol(v_{i-1}, v_i) \cdot \overline{\Theta} (v_i,
v_{i-1})}
{-\mu(r) + \mu(v_{i})}\\
\end{array}
\]
for every $r \in V_\Gamma$ by Lemma \ref{lemma_super_perfect}.
Therefore, we have
\begin{equation} \label{equation: product higher dim}
\tau_{p_k}^+ (r) \cdot \tau_{q_j}^- (r) = \Lambda_r \cdot
\sum_{\mathbf{v} \in \sum_{p_k}^{q_j}(r)} \,
\frac{\prod_{i=1}^{n-l} \, \vol(v_{i-1}, v_i) \cdot
\Theta (v_{i-1}, v_i)} {\prod_{i \in \{ 0, 1, \cdots,
n-l \} \setminus \{ c_r \}} \, \big[ \,
-\mu(r)+\mu(v_i) \, \big]},
\end{equation}
since $|\mathbf{v}| = n-l$ for every $\mathbf{v} \in \Sigma_{p_k}^{q_j}(r)$ and each $\mathbf{v}$ is
of the form
\[
    \mathbf{v} = (v_0=p_k, \cdots, v_{c_r}=r, \cdots,
v_{n-l}=q_j).
\]
Eventually, by applying the ABBV-localization theorem (Theorem \ref{theorem_localization}), we have
\begin{align}
\label{equation: calculation higher dim} a_{jk} &=
\langle  [\omega]^{n-l} \wedge f^* (\tau_{p_k}^+)
\wedge f^* (\tau_{q_j}^-), [M]  \rangle \\
\notag &= \int_M ~ \Big[\prod_{i=1}^{n-l} [\widetilde{\omega}_i]\Big] \cdot
\tau_{p_k}^+ \cdot \tau_{q_j}^-  \\
\notag&= \sum_{r \in V_\Gamma} ~ \Big(
\Big[\prod_{i=1}^{n-l} [\widetilde{\omega}_i]\Big] \cdot \tau_{p_k}^+ \cdot
\tau_{q_j}^- \Big) (r) \Big/ \Lambda_r,
\end{align}
where $\widetilde{\omega}_i$ is any equivariant symplectic form for each $i=1, \cdots, n-l$.
Note that \eqref{equation: calculation higher dim} is equal to
\[
\sum_{r \in V_\Gamma} ~ \Big[\prod_{i=1}^{n-l} [\widetilde{\omega}_i]\Big](r)
\cdot \Bigg[ ~ \sum_{\mathbf{v} \in
\sum_{p_k}^{q_j}(r)} \,  \frac{\prod_{i=1}^{n-l} \,
\vol(v_{i-1}, v_i) \cdot \Theta (v_{i-1}, v_i)}
{\prod_{i \in \{ 0, 1, \cdots, n-l \} \setminus \{ c_r
\}} \, \big[ \, -\mu(r) + \mu(v_i) \, \big]} ~ \Bigg]
\]
by \eqref{equation: product higher dim}. Also, note that
$[\widetilde{\omega}_i](r) = [\widetilde{\omega}_i]|_r = -\mu(r) + d_i$ for some
$d_i \in \mathfrak{t}^*.$ Since \eqref{equation: calculation higher dim} holds for any choice of $\omega_i$, each $d_i$ can be
chosen arbitrarily. Therefore, the coefficient
$a_{jk}$ is equal to
\[
\quad \sum_{r \in V_\Gamma} \quad \Bigg[ ~ \prod_{i=1}^{n-l} \,
\big[ \, \mu(r) - d_i \, \big] ~ \Bigg] \cdot \Bigg[ ~
\sum_{\mathbf{v} \in \sum_{p_k}^{q_j}(r)} ~ \frac{\prod_{i=1}^{n-l}
\, \big[ \, \vol(v_{i-1},v_i) \cdot \Theta (v_{i-1}, v_i) \, \big]}
{\prod_{i \in \{ 0, 1, \cdots, n-l\} \setminus \{ c_r \}} \, \big[
\, \mu(r)-\mu(v_i) \, \big]} ~ \Bigg].
\]
This finishes the proof.
\end{proof}

From Proposition \ref{proposition_coeff_higher_dim}, we can obtain the following.

\begin{corollary}\label{corollary_index_differ_by_two}
    Suppose that $n-l = 1$. Then 
\begin{enumerate}
\item $a_{jk} = \left\{
                \begin{array}{ll}
                \Theta (p_k, q_j) \cdot \vol (p_k, q_j)  & \text{ if }
                (p_k, q_j) \in E_\Gamma, \\
                0   & \text{ if }
                (p_k, q_j) \not \in E_\Gamma,~\text{and}
                          \end{array}
                        \right.$
        \item $a_{jk}$ is nonzero if and only if $(p_k, q_j) \in E_\Gamma$.
\end{enumerate}
\end{corollary}

\begin{proof}
Suppose that $n-l = 1$ and let $p = p_k$ (resp. $q = q_j$) be any index $l$ (resp. index $2n- l$) vertex in $V_\Gamma$. If $p$ and $q$ are
not adjacent, then $a_{jk} = 0$ by Proposition
\ref{proposition_coeff_higher_dim} since $\sum_p^q$
is empty. If $p$ and $q$ are adjacent, the formula of
Proposition \ref{proposition_coeff_higher_dim} is
reduced to
    \[
        \begin{array}{ccl}
        a_{jk} & = & \displaystyle (\mu(p) - d) \cdot \frac{\vol(p,q) \cdot \Theta(p,q)}{\mu(p) - \mu(q)} + (\mu(q) - d) \cdot\frac{\vol(p,q) \cdot \Theta(p,q)}{\mu(q) - \mu(p)}\\
                   & = & \vol(p,q) \cdot \Theta(p,q)
        \end{array}
    \]
for any choice of $d \in \mathfrak{t}^*$. The second statement (2) easily follows from (1) and 
the fact that $\vol(p,q)$ and $\Theta(p,q)$ are both nonzero for every $(p,q) \in E_\Gamma$.
\end{proof}

\subsection{Concluding remark}
\label{ssecConcludingRemark}

Sabatini and Tolman \cite[Theorem 0.3]{ST} gave a generalized formula of Theorem \ref{theorem_gt}.
We state the modified version of the theorem which fits in our context as follows.

\begin{theorem}\label{theorem_ST}\cite{ST}
    Let $(M,\omega,\mu)$ be a Hamiltonian GKM $\, T$-manifold such that
    the corresponding GKM graph $\Gamma$ is index increasing with respect to some generic vector $\xi \in \frak{t}$.
    Let $p$ and $q$ be any two fixed point.
    For each fixed point $z \in M^T$, let $w_z$ be any element in $H^2_T(M)$ such that
    $w_z(q) \neq w_z(z)$. Then
    \[
        \tau^+_p(q) = \Lambda^+_q \cdot \Bigg[ ~ \sum_{\mathbf{v} \in \sum_{p}^{q}} ~ \prod_{i=1}^{|\mathbf{v}|} \frac{w_{v_i}(v_{i+1}) - w_{v_i}(v_i)}
        {w_{v_i}(q) - w_{v_i}(v_i)} \cdot \frac{\tau^+_{v_i}(v_{i+1})}{\Lambda^+_{v_{i+1}}}
         ~ \Bigg].
    \]
\end{theorem}

\begin{remark}
We can easily see that Theorem \ref{theorem_ST} is a generalization of Theorem \ref{theorem_gt}
by taking $w_z = [\widetilde{\omega}_\mu]$ for every $z \in M^T$.
\end{remark}

We expect that Theorem \ref{theorem_ST}, together with the flexibility of the choice of
$d_i$'s in Proposition \ref{proposition_coeff_higher_dim}, may provide a more simple formula of Proposition \ref{proposition_coeff_higher_dim}.
Indeed, the coefficients and the determinant of $A_l(M, \omega)$ can be expressed by very simple formulas
in the following special case \cite{CK2}. More precisely, the authors proved in \cite{CK2} that if there exists a vector $\xi \in \mathfrak{t}$ such that
$\mu_\xi (v) = \lambda_v$ for each fixed point $v,$ i.e., $\mu_\xi$ is a self-indexing moment map,
then $(M,\omega)$ satisfies the hard Lefschetz property by computing the determinant of $A_l(M,\omega)$ for each $l$.

\bigskip

\section{Six-dimensional Hamiltonian GKM manifolds with index increasing graphs}
\label{secSixDimensionalHamiltonianGKMManifoldsWithAnIndexIncreasingGraph}

In this section, we restrict our attention to six-dimensional Hamiltonian GKM manifolds and give the proof of Theorem
\ref{theorem_main}. 

Let $(M,\omega)$ be a six-dimensional closed
symplectic manifold and let $T$ be a two-dimensional compact torus
acting on $(M,\omega)$. Assume that the $T$-action is Hamiltonian
GKM with a moment map $\mu: M \longrightarrow \mathfrak{t}^*$. Let
$\xi \in \mathfrak{t}$ be a generic vector having rational slop such
that the corresponding GKM graph $(\Gamma, V_\Gamma, E_\Gamma)$ is
index increasing with respect to $\xi$. Note that the vector $\xi$
defines a circle subgroup $S^1$ of $T$ and the induced $S^1$-action
on $(M,\omega)$ is Hamiltonian with respect to a moment map 
$\mu_\xi = \langle \mu, \xi \rangle$. 
We start with the following well-known fact.

\begin{lemma}\label{lemma_odd_betti_vanish}\cite{Au}
    $b_{\mathrm{odd}}(M) = 0$.
\end{lemma}
\begin{proof}
    See Remark \ref{remark_Ham_circle_action_property}.
\end{proof}

We reformulate Theorem \ref{theorem_main}
by using equivariant Thom classes defined in Section \ref{secGraphCohomologyAndEquivariantCohomologyOfHamiltonianGKMManifolds}.
Recall that $(M, \omega)$ satisfies
the hard Lefschetz property if and only if the
Hodge-Riemann bilinear form $\HR_l$ is nondegenerate for every
$l=0, \cdots, 3$, see Section \ref{secIntroduction}.
It is straightforward that
\begin{displaymath}
\begin{array}{cccc}
\HR_0 : & H^0(M)  \times  H^0(M) & \longrightarrow &
\R \\ [0.5em] & (\alpha, \beta) & \longmapsto &
<\alpha \beta {[\omega]}^3, [M]>
\end{array}
\end{displaymath}
is nondegenerate since $\omega^3$ is a volume form on
$M$. Therefore, by Lemma \ref{lemma_odd_betti_vanish},
$\HR_2$ is non-degenerate if and only if $(M,\omega)$ satisfies the hard Lefschetz property.

Now, let $\tau^+_v$ and
$\tau^-_v$ be the equivariant Thom classes for each
vertex $v \in V_\Gamma$ with respect to $\xi$ and
$-\xi,$ respectively. Let $b_2 := b_2(M)$ be the
second Betti number of $M$ and let
\[
\{p_1, \cdots, p_{b_2}\} \qquad \text{and}
\qquad \{q_1, \cdots, q_{b_2}\}
\]
be the sets of index-two vertices and index-four
vertices, respectively. These sets have the same
number of elements by the Poincar\'{e} duality. Let
$x_1$ and $x_2$ be the generators of 
$S(\mathfrak{t}^*) \cong H^*(BT) = \R[x_1,x_2]$ given in \eqref{equation_polynomial_ring}. 
By Theorem \ref{theorem_GKM},
we may identify $H^*_T (M)$ with $H(\Gamma, \alpha)$
and the set of all equivariant Thom classes forms a
basis of $H^*_T(M)$ as an $H^* (BT)$-module by Theorem
\ref{theorem_Thom_class}. In particular, each of
\[
\{~ x_1, x_2, \tau^+_{p_k} ~|~ 1 \le k \le
b_2 ~\} \qquad \text{and} \qquad \{ ~ x_1,
x_2, \tau^-_{q_j} ~|~ 1 \le j \le b_2 ~\}
\]
becomes a basis of $H^2_T(M)$ as an $\R$-vector space.

\begin{lemma} \label{lemma_nonsingular}
Let $f : M \hookrightarrow M \times_T ET$ be an
inclusion of a fiber $M$ 
given in \eqref{diagram : M-bundle over BT}
and let $f^* :
H^*_T(M) \rightarrow H^*(M)$ be its induced ring
homomorphism. Then $\HR_2$ is nondegenerate if and only
if the $b_2 \times b_2$ matrix \[ A_2 (M,\omega) = (a_{jk})_{1 \le j, \,
k \le b_2} := \Big( \HR_2(f^*\tau_{p_k}^+, f^*
\tau_{q_j}^-) \Big)_{1 \le j, \, k \le b_2}\] is
nonsingular.
\end{lemma}
\begin{proof}
Recall that $\mathcal{B}_2^+ = \{f^*\tau^+_{p_1}, \cdots, f^*\tau^+_{p_{b_2}}\}$ and 
$\mathcal{B}_2^- = \{f^*\tau^-_{q_1}, \cdots, f^*\tau^-_{q_{b_2}}\}$ are bases of $H^2(M)$
by Theorem \ref{theorem_Thom_class} and Theorem \ref{theorem_equivariant_formality}.  
Then $A_2(M,\omega)$ is the matrix 
representing $\mathrm{HR}_2$ with respect to the pair $(\mathcal{B}_2^+, \mathcal{B}_2^-)$ and this finishes the proof.
\end{proof}

Using  Lemma \ref{lemma_nonsingular}, we can reformulate Theorem \ref{theorem_main} into the following proposition.

\begin{proposition}[Theorem \ref{theorem_main}]\label{proposition_main}
The matrix $A_2 (M,\omega)$ is nonsingular.
\end{proposition}

Let $o$ (resp. $r$) be the unique index-zero (resp. index-six) vertex of $\Gamma$.
Recall that $\vol$ and $\Theta$ are functions on the edge set $E_\Gamma$ defined by 
\[
	\vol(p,q) := \Big( \mu(q)-\mu(p) \Big) \Big/ \alpha (p,q) \, \in \, Q(\mathfrak{t}^*) 
\] and 
\[
	\Theta(p,q) := \displaystyle \frac{\rho_{\alpha(p,q)} (\Lambda^+_p)}{\rho_{\alpha(p,q)} (\Lambda^+_{q} / \alpha(q,p))} \in \, Q(\mathfrak{t}^*) 
\] for $(p,q) \in E_\Gamma$. See Section \ref{ssecVolAndTheta}. 
The following proposition is straightforward by Corollary \ref{corollary_index_differ_by_two}.

\begin{proposition} \label{proposition_matrix_coefficient}
Let $A_2 (M,\omega) = (a_{jk})_{1 \le j,k \le b_2}$ be given in Lemma \ref{lemma_nonsingular}. 
Then 
\begin{enumerate}
\item $a_{jk} = \Theta (p_k, q_j) \cdot \vol (p_k, q_j),$ and 
\item $a_{jk}$ is nonzero if and only if $p_k$ and $q_j$ are adjacent.
\end{enumerate}
\end{proposition}

Suppose that $(p,q) \in E_\Gamma$ where $p$ (resp. $q$) is an index-two (resp. index-four) vertex of $\Gamma$.
Then there exists a unique vertex $v \neq p$ adjacent to and below $q$ with respect to $\xi$,
that is, $\mu_\xi(v) < \mu_\xi (q).$
 Also, by the index increasing property, the index of $v$ is less than or equal to two.  
Since
\begin{equation} \label{equation: supp of Thom
class} \supp \tau_p^+ \subset \{ p \} \cup \{
\text{index-4 vertices adjacent to } p \} \cup \{
\text{the index-6 vertex } r \}
\end{equation}
by Theorem \ref{theorem_Thom_class}, we have $\tau_p^+(v) = 0$, and therefore 
$
\tau_p^+(q) = k \cdot \alpha (q, v) $
for some real number $k \in \R$ by Lemma \ref{lemma_zero_at_a_vertex}.(1). Furthermore, since 
$\tau_p^+(q) - \tau_p^+(p) \equiv 0$ modulo $\alpha(p,q)$, we can easily see that $k = \Theta(p,q)$, 
and therefore 
 \begin{equation} \label{equation: definition of Theta}
\tau_p^+(q) = \Theta(p,q) \cdot \alpha (q, v)
\end{equation}
see Figure \ref{figure_GT_def} and Figure \ref{figure: proof of lemma easy condition}.(a).
Note that if $p$ and $q$ are not adjacent, then $\tau_p^+(q) = 0$ by \eqref{equation: supp of Thom class}. 
Thus we obtain the following lemma. See also \cite[Theorem 4.1]{GT}.

\begin{lemma}\label{lemma_condition_for_nonzero}
Let $p$ and $q$ be an index-two and index-four vertices, respectively. Then $p$ and $q$ are adjacent if and only if $\tau^+_p(q) \neq 0$.
\end{lemma}

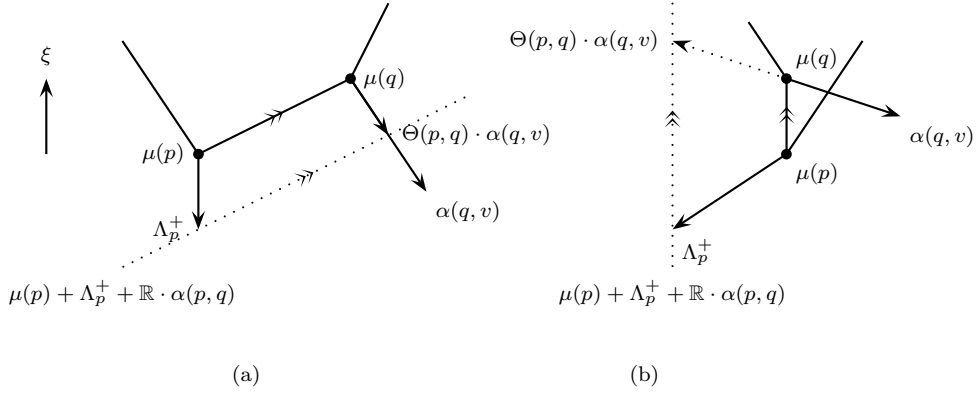
\begin{figure}[ht]
\begin{center}
\mbox{\subfigure[ ]{
\begin{pspicture}(-1.5,-2.5)(3,2) \footnotesize
\psline(-1,1.5)(0,0)(2,1)(2.5,2)
\psline[arrowsize=5pt]{->}(0,0)(0,-1)
\psline[arrowsize=5pt]{->}(2,1)(2.5,0.25)
\psline[arrowsize=5pt]{->}(2,1)(3,-0.5)
\psdots[dotsize=4pt](0,0)(2,1)
\psline[linestyle=dotted](-1,-1.5)(3.5,0.75)

\psline[arrowsize=5pt]{->}(-2,0)(-2,1)
\uput[u](-2,1){$\xi$}

\psline[arrowsize=6pt,arrowlength=0.6,arrowinset=0.7,
linestyle=none]{<}(0.9,0.45)(2,1)
\psline[arrowsize=6pt,arrowlength=0.6,arrowinset=0.7,
linestyle=none]{<}(1,0.5)(2,1)

\psline[arrowsize=6pt,arrowlength=0.6,arrowinset=0.7,
linestyle=none]{<}(1.3,-0.35)(2.5,0.25)
\psline[arrowsize=6pt,arrowlength=0.6,arrowinset=0.7,
linestyle=none]{<}(1.4,-0.3)(2.5,0.25)

\uput[l](0,0){$\mu(p)$} \uput[r](2,1){$\mu(q)$}
\uput[l](0,-1){$\Lambda_p^+$} \uput[dr](3,-0.5){$\alpha
(q, v)$} \uput[d](-1,-1.5){$\mu(p) + \Lambda_p^+ + \R
\cdot \alpha (p, q)$} \uput[r](2.5,0.25){$\Theta (p,
q) \cdot \alpha (q, v)$
}
\end{pspicture}}

\subfigure[ ]{
\begin{pspicture}(-3,-2)(2.5,2.5) \footnotesize
\psline(2.5,2)(1.5,0.5)(1.5,1.5)(1,2.25)
\psline[arrowsize=5pt]{->}(1.5,0.5)(0,-0.5)
\psline[arrowsize=5pt]{->}(1.5,1.5)(3,1)
\psline[linestyle=dotted,
arrowsize=5pt]{->}(1.5,1.5)(0,2)
\psdots[dotsize=4pt](1.5,0.5)(1.5,1.5)
\psline[linestyle=dotted](0,-1)(0,2.5)

\psline[arrowsize=6pt,arrowlength=0.6,arrowinset=0.7,
linestyle=none]{<}(0,0.85)(0,2.5)
\psline[arrowsize=6pt,arrowlength=0.6,arrowinset=0.7,
linestyle=none]{<}(0,0.95)(0,2.5)

\psline[arrowsize=6pt,arrowlength=0.6,arrowinset=0.7,
linestyle=none]{<}(1.5,0.9)(1.5,1.5)
\psline[arrowsize=6pt,arrowlength=0.6,arrowinset=0.7,
linestyle=none]{<}(1.5,1.0)(1.5,1.5)

\uput[dr](1.5,0.5){$\mu(p)$}
\uput[ur](1.5,1.5){$\mu(q)$}
\uput[dr](0,-0.5){$\Lambda_p^+$} \uput[dr](3,1){$\alpha
(q, v)$} \uput[d](0,-1){$\mu(p) + \Lambda_p^+ + \R \cdot
\alpha (p, q)$} \uput[l](0,2){$\Theta (p, q) \cdot
\alpha (q, v)$}
\end{pspicture}}
}
\end{center}
\caption{ \label{figure: proof of lemma easy
condition} (a) $\Theta(p,q) > 0$, (b) $\Theta(p,q) < 0$}
\end{figure}

\subsection{Positivity of $\Theta$}
The positivity of $\Theta (p, q)$ will play an essential role
for proving the non-singularity of $A_2(M,\omega)$.

\begin{definition} \label{definition: same side}
A subset of a real two-dimensional vector space $V$ is said to be
{\em in the same side with respect to a straight
line $L$} in $V$ if it is contained in the closure of a
connected component of $V-L.$
\end{definition}

Let $(p,q) \in E_\Gamma$ for an index-two vertex $p$ and an index-four
vertex $q$, respectively. By definition of graph cohomology, we have
\begin{equation} \label{equation: graph cohomology
element} \tau_p^+(p) \equiv \tau_p^+(q) \mod \alpha
(p, q).
\end{equation}
Substituting
\[
\tau_p^+(p) = \Lambda_p^+ \quad \text{ and } \quad
\tau_p^+(q) = \Theta (p, q) \cdot \alpha (q, v)
\]
into (\ref{equation: graph cohomology element}), we have
\begin{equation} \label{equation: Theta is
positive} -\Lambda_p^+ + \Theta (p, q) \cdot \alpha (q,v)
= k \cdot \alpha (p, q) \quad \text{for some real
number } k.
\end{equation}
Adding $\mu(q) - \mu(p) = \vol (p, q) \cdot \alpha
(p,q)$ to both sides of \eqref{equation: Theta is positive},
we have
\[
\Big( -\mu(p) - \Lambda_p^+ \Big) + \Big(\mu(q) + \Theta
(p, q) \cdot \alpha (q,v) \Big) = k^\prime \cdot
\alpha (p, q) \quad \text{where} ~k^\prime = k + \vol
(p, q).
\]
This implies that
\[
\mu(q) + \Theta
(p, q) \cdot \alpha (q, v)  \in  \Big(\mu(p) + \Lambda_p^+ \Big)  + \R \cdot
\alpha (p, q),
\]
that is, $\mu(q) + \Theta (p, q) \cdot \alpha (q, v)$ is contained in the dotted line in Figure \ref{figure: proof of lemma easy condition}. 

On the other hand, $\Theta$ can be understood in the following way :
the straight line $\mu(q)
+ \R \cdot \alpha (q, v)$ intersects $\mu(p) + \Lambda_p^+
+ \R \cdot \alpha (p, q)$ at $\mu(q) + \Theta (p, q)
\cdot \alpha (q, v),$ see Figure \ref{figure: proof of
lemma easy condition} in which the line segment
connecting $\mu(p)$ and $\mu(q)$ is parallel with the dotted straight
line marked by the doubled arrow vector.
Consequently, we can see that
$\mu(p) + \Lambda_p^+$ and $\mu(q) + \Theta (p, q) \cdot
\alpha (q, v)$ are in the same side with respect to
the straight line $\mu(p) + \R \cdot \alpha (p, q).$
Equivalently, $\Lambda_p^+$ and $\Theta (p, q) \cdot \alpha (q, v)$ are in the same side with respect to
the straight line $\R \cdot \alpha (p, q).$ 
From this observation, we deduce the following. 
\begin{lemma}\label{lemma_condition_for_positive} Two vectors
$\Lambda_p^+$ and $\alpha (q, v)$ are in the same side
with respect to the straight line $\R \cdot \alpha (p,
q)$ if and only if $\Theta (p, q)$ is positive.
\end{lemma}

As the following examples show, 
Lemma \ref{lemma_condition_for_positive} enables us to check the positivity of $\Theta$ by looking up the shape of a graph.

\begin{example} \label{example: negative Theta}
For example, $\Theta (p, q)$ is positive in Figure \ref{figure: proof of lemma easy condition}.(a).
On the other hand, in Figure \ref{figure: proof of lemma easy condition}.(b), $\Theta (p, q)$ is negative since two vectors $\Lambda_p^+$ and
$\alpha (q, v)$ are not in the same side with respect
to the straight line $\R \cdot \alpha (p, q).$ 

More concrete examples are as
follows.
In Figure \ref{figure: image of moment}.(d),
$\Theta(p, q)$ is negative for the index-two vertex
$p$ and the index-four vertex $q$ which lie on the
interior of the moment map image. In fact, Figure
\ref{figure: image of moment}.(d) corresponds to
Tolman's example of a non-K\"{a}hler Hamiltonian GKM
manifold explained in Example \ref{example_tolman}. On
the contrary, $\Theta (p, q)$ is positive for any $(p,q) \in E_\Gamma$ 
in Figure \ref{figure: image of
moment} with $\mathrm{ind} ~p = 2$ and $\mathrm{ind} ~q = 4$.
\end{example}

By using Lemma \ref{lemma_condition_for_positive}, we
can state a more refined condition under which $\Theta
(p, q)$ is positive. For an index-two vertex $p,$ we
denote by $\gamma_p$ the cycle whose vertices consist
of $p$ and vertices connected by ascending
paths starting at $p,$ and call it the
{\em ascending cycle} starting at $p.$ In other
words, the set of all vertices contained in $\gamma_p$
is the right hand side of (\ref{equation: supp of Thom
class}). Note that $p$ is of index-two so that $p$
should be adjacent to at least one and at most two
index-four vertices by the three valency of $\Gamma$, see Section \ref{ssecGKMManifolds}.
In particular, $\gamma_p$ has three or four vertices. An ascending cycle is
called {\em triangular} (resp. {\em tetragonal}) if it
has three (resp. four) vertices. In other words, $\gamma_p$
is triangular if and only if $p$ is adjacent to
exactly one index-four vertex and to $r$. Also,
$\gamma_p$ is tetragonal if and only if $p$ is
adjacent to exactly two index-four vertices.

\begin{example}
Let us consider examples of ascending cycles in Figure \ref{figure: image of moment}.
Each of (a) and (b) has one triangular and no tetragonal ascending cycle. And each
of (c), (e), and (f) has one triangular and one tetragonal
ascending cycles. Each of (d) and (g) has no triangular ascending cycle and it has two tetragonal ascending cycles.
And (h) has no triangular ascending cycle and has three tetragonal ascending cycles.

\end{example}

For a tetragonal ascending cycle $\gamma_p$ starting
at $p,$ we denote by $\square \gamma_p$ the union of
images $\mu(S_e^2)$ for edges $e$ of $\gamma_p.$ Thus
$\square \gamma_p$ is a tetragon in $\mathfrak{t}^*.$
It is classical that tetragons are classified into
three types as follows, see Figure \ref{figure:
tetragons}.

\begin{lemma}\cite[p.50]{We} \label{lemma: tetragon}
Tetragons $\square ABCD$ in a plane are classified
into three types :
\begin{enumerate}
\item[(a)] convex {\em if for each edge $\ell$ of $\square ABCD$, $\{A, B, C, D\}$
is in the same side with respect to the straight
line generated by $\ell$,}
\item[(b)] concave {\em if the convex hull $\Conv \{A, B, C, D\}$
is triangular, i.e., a vertex is contained in the
interior of $\Conv \{A, B, C, D\},$}
\item[(c)] crossed {\em if there exist two opposite
line segments passing through each other.}
\end{enumerate}
\end{lemma}

\begin{figure}[ht]
\begin{center}
\mbox{

\subfigure[convex]{
\begin{pspicture}(-0.8,-0.2)(2.8,2)\footnotesize
\psline(0,0)(2,0.5)(1.5,2)(0.5,2)(0,0)
\uput[l](0,0){$A$} %
\uput[r](2,0.5){$B$} %
\uput[r](1.5,2){$C$} %
\uput[l](0.5,2){$D$}  %
\end{pspicture}}

\subfigure[concave]{
\begin{pspicture}(-0.8,-0.2)(2.8,2)\footnotesize
\psline(0,0)(2,1.5)(0,2)(0.5,1)(0,0)
\uput[l](0,0){$A$} %
\uput[r](2,1.5){$B$} %
\uput[l](0,2){$C$} %
\uput[l](0.5,1){$D$}  %
\end{pspicture}}

\subfigure[crossed]{
\begin{pspicture}(-0.8,-0.2)(2.8,2)\footnotesize
\psline(0,0)(2,0.5)(0,2)(2,2)(0,0)
\uput[l](0,0){$A$} %
\uput[r](2,0.5){$B$} %
\uput[l](0,2){$C$} %
\uput[r](2,2){$D$}  %
\end{pspicture}}

}
\end{center}
\caption{\label{figure: tetragons} Three types of
tetragons}
\end{figure}
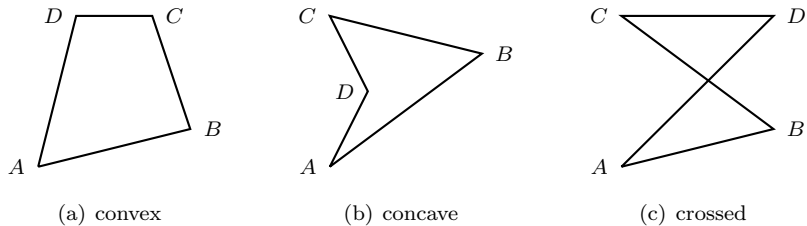

Similarly, we call a tetragonal ascending cycle
$\gamma_p$ {\em convex, concave,} or
{\em crossed} if the tetragon $\square \gamma_p$ is
convex, concave, or crossed, respectively. We
introduce a new condition which guarantees that $\Theta (p,
q)$ is positive.

\begin{proposition}
\label{proposition_positive_if_convex} For an adjacent
index-two vertex $p$ and an index-four vertex $q,$ if
the ascending cycle $\gamma_p$ is tetragonal and
convex, then $\Theta (p, q)$ is positive.
\end{proposition}

Before we prove Proposition \ref{proposition_positive_if_convex}, we give the following lemma without proof,
which is essentially the same as Lemma \ref{lemma_GKM_weight_pair}.

\begin{lemma} \label{lemma_GKM_weight_pair_2}
For each oriented edge $e \in E_\Gamma,$ we can
reorder $\alpha_{j, i(e)}$'s and $\alpha_{j, t(e)}$'s
so that (1) $\alpha_{n, t(e)} = -\alpha_{n, i(e)} =
-\alpha (e),$ and (2) $\alpha_{j, t(e)},$ $\alpha_{j,
i(e)}$ are in the same side with respect to $\R \cdot
\alpha (e)$ for each $1 \le j \le n-1.$
\end{lemma}

\begin{proof}[Proof of Proposition \ref{proposition_positive_if_convex}]
Pick the vertex $v \ne p$ which is adjacent to and
below $q$. By the assumption,
there exists another index-four vertex $q^\prime \ne
q$ which is adjacent to and above $p$, see Figure \ref{figure: positive if convex}.
Since $\square \gamma_p$ is convex, $\alpha (p,
q^\prime)$ and $\alpha (q, r)$ are in the same side
with respect to $\R \cdot \alpha (p, q)$ by Lemma
\ref{lemma: tetragon}. Applying Lemma
\ref{lemma_GKM_weight_pair_2} to the edge $(p, q),$
two weights $\Lambda_p^+ = \alpha (p, o)$ and $\alpha
(q, v)$ should be in the same side with respect to $\R
\cdot \alpha (p, q).$ Therefore, $\Theta (p, q)$ is
positive by Lemma \ref{lemma_condition_for_positive}.
\end{proof}

\begin{figure}[ht]
\begin{center}
\begin{pspicture}(-2,-1)(2.5,2) \footnotesize
\psline(-0.66,1)(0,0)(2,1)(2.5,2)(-0.66,1)
\psline(2,1)(2.5,0.25) \psline(0,0)(0,-0.5)
\psdots[dotsize=4pt](-0.66,1)(0,0)(2,1)(2.5,0.25)(0,-0.5)(2.5,2)

\psline[arrowsize=5pt]{->}(-2.5,0)(-2.5,1)
\uput[u](-2.5,1){$\xi$}

\uput[l](0,0){$\mu(p)$}
\uput[r](2,1){$\mu(q)$}
\uput[u](-0.86,1){$\mu(q^\prime)$}
\uput[r](2.4,2.2){$\mu(r)$}
\uput[r](2.5,0.25){$\mu(v)$}
\uput[d](0,-0.5){$\mu(o)$}
\end{pspicture}
\end{center}
\caption{ \label{figure: positive if convex} Proof of
Proposition \ref{proposition_positive_if_convex}}
\end{figure}
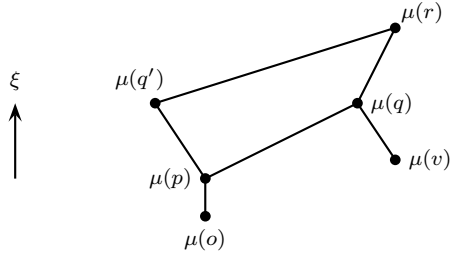

\begin{example}
\label{example: negative Theta 2} Let us consider
Figure \ref{figure: image of moment}.(d). For the
index-two vertex $p$ in the interior of $\mu(M)$,
$\gamma_p$ is not convex but concave. As we have seen in Example \ref{example: negative Theta}, 
$\Theta(p,q)$ is negative for the index-four vertex $q$ in the interior of $\mu(M)$.
On the other hand, for any other index-two vertex $p$ in Figure 
\ref{figure: image of moment}, if $\gamma_p$ is tetragonal, it is convex.  
\end{example}

\subsection{Weak classification}
In addition to Proposition \ref{proposition_matrix_coefficient} and Proposition \ref{proposition_positive_if_convex},
we need to understand the GKM
graph more precisely to show that the determinant of
the matrix $A_2(M,\omega)$ is nonzero.
Let $\mathcal{E}$ and $\mathcal{V}$ be the numbers of
non-oriented edges and vertices of $\Gamma,$
respectively.

\begin{lemma}\label{lemma_vert_edge_relation}
Let $\mathcal{V}$ and $\mathcal{E}$ be given above. Then
\begin{itemize}
    \item $2\mathcal{E}=3\mathcal{V}$, and
    \item the number of index-two vertices, i.e., $b_2$ is equal to $\mathcal{V}/2-1$.
\end{itemize}
\end{lemma}
\begin{proof}
    The first statement follows from the three valency of $\Gamma.$
    Also, the second statement follows from the Poincar\'{e} duality.
\end{proof}

The following
proposition classifies all possible GKM graphs
$\Gamma$ into eight types according to the following
four criteria :
\begin{enumerate}
\item the shape of the moment map image $\mu(M),$
\item the number of vertices of $\Gamma,$
\item adjacency between $o$ and $r,$ and
\item the number of tetragonal ascending
cycles starting at index-two vertices.
\end{enumerate}

\begin{proposition} \label{proposition: graph shape}
Let $(M,\omega)$ be a six-dimensional closed Hamiltonian $T^2$-manifold.
Suppose that the action is GKM and its GKM graph $\Gamma$ is index-increasing with respect to 
some generic $\xi \in \frak{t}$. Then
$\Gamma$ satisfies one of (a)$\sim$(h) of Table \ref{table: graph shape}.
\begin{table}[ht!]
\begin{center}
{\footnotesize
\begin{tabular}{c||c|c|c|c}
      & $\mu(M)$  & $\mathcal{V}$  & $o$ is adjacent to $r$?   & the number of tetragonal ascending cycles   \\
      &           &                &                           & starting at an index-two vertex   \\
 \hhline{=#=|=|=|=}

 (a)  & triangle  & 4              & Yes                       & 0                                 \\
 (b)  & tetragon  & 4              & Yes                       & 0                                 \\
 (c)  & tetragon  & 6              & No                        & 1                                 \\
 (d)  & tetragon  & 6              & Yes                       & 2                                 \\
 (e)  & pentagon  & 6              & No                        & 1                                 \\
 (f)  & hexagon   & 6              & No                        & 1                                 \\
 (g)  & hexagon   & 6              & Yes                       & 2                                 \\
 (h)  & hexagon   & 8              & No                        & 3
\end{tabular}}
\quad \\ \quad \\ \quad \\
\caption{\label{table: graph shape} Eight types of
possible index increasing GKM graphs}
\end{center}
\end{table}
\end{proposition}

\begin{figure}[ht]
\begin{center}
\mbox{\subfigure[]{
\begin{pspicture}(-0.4,-0.2)(2.4,2)\footnotesize


\pspolygon[fillstyle=solid,fillcolor=lightgray](1,2)(0,0)(2,0.5)
\psline(1,2)(1,1)(0,0)(1,1)(2,0.5)
\psdots[dotsize=4pt](1,2)(1,1)(0,0)(2,0.5) \uput[l](1,
2){$\mu(r)$} \uput[l](0,0){$\mu(o)$}
\psline[arrowsize=6pt]{->}(-0.5,1.3)(-0.5,2)
\uput[l](-0.5,2){$\xi$}
\end{pspicture}}

\subfigure[]{
\begin{pspicture}(-0.4,-0.2)(2.4,2)\footnotesize
\pspolygon[fillstyle=solid,fillcolor=lightgray](1,0)(0,0.65)(0.75,2)(2,1.5)
\psline(1,0)(0.75,2) \psline(0,0.65)(2,1.5)
\psdots[dotsize=4pt](1,0)(0,0.65)(0.75,2)(2,1.5)
\uput[r](1, 0){$\mu(o)$} \uput[l](0.75,2){$\mu(r)$}
\end{pspicture}}

\subfigure[]{
\begin{pspicture}(-0.4,-0.2)(2.4,2)\footnotesize
\pspolygon[fillstyle=solid,fillcolor=lightgray](1.5,0)(2,1.75)(0.75,2)(0,0.5)
\psline(1.5,0)(1.4,1)(0.9,1.25)(0,0.5)
\psline(2,1.75)(1.4,1)(0.9,1.25)(0.75,2)
\psdots[dotsize=4pt](1.5,0)(2,1.75)(0.75,2)(0,0.5)(1.4,1)(0.9,1.25)
\uput[r](1.5, 0){$\mu(o)$} \uput[l](0.75,2){$\mu(r)$}
\end{pspicture}}

\subfigure[]{
\begin{pspicture}(-0.4,-0.2)(2.4,2)\footnotesize
\pspolygon[fillstyle=solid,fillcolor=lightgray](0,0)(2,0.5)(2,1.5)(0,2)
\psline(0,0)(0.5,0.65)(0.5, 1.25)(0,2)
\psline(2,0.5)(0.5, 1.25)(0.5,0.65)(2,1.5)
\psdots[dotsize=4pt](0,0)(2,0.5)(2,1.5)(0,2)(0.5,0.65)(0.5,
1.25) %
\uput[l](0, 0){$\mu(o)$} \uput[l](0,2){$\mu(r)$}
\end{pspicture}}
}

\mbox{

\subfigure[]{
\begin{pspicture}(-0.4,-0.2)(2.4,2.2)\footnotesize
\pspolygon[fillstyle=solid,fillcolor=lightgray](0.5,2)(1.75,1.75)(2,0.6)(1.25,0)(0,0.75)
\psline(0.5,2)(0.75,1.2)(2,0.6)(0.75,1.2)(0,0.75)
\psline(1.75,1.75)(1.25,0)
\psdots[dotsize=4pt](0.5,2)(0.75,1.2)(1.75,1.75)(2,0.6)(1.25,0)(0,0.75)
\uput[l](0.5, 2){$\mu(r)$} \uput[r](1.25,0){$\mu(o)$}
\end{pspicture}}

\subfigure[]{
\begin{pspicture}(-0.4,-0.2)(2.4,2.2)\footnotesize
\pspolygon[fillstyle=solid,fillcolor=lightgray](1,0)(1.75,0.5)(2,1.3)(1.3,2)(0.25,1.50)(0.1,0.5)
\psline(1,0)(0.25,1.5) \psline(0.1,0.5)(2,1.3)
\psline(1.75,0.5)(1.3,2)
\psdots[dotsize=4pt](1,0)(1.75,0.5)(2,1.3)(1.3,2)(0.25,1.5)(0.1,0.5)
\uput[l](1, 0){$\mu(o)$} \uput[r](1.3,2){$\mu(r)$}
\end{pspicture}}

\subfigure[]{
\begin{pspicture}(-0.4,-0.2)(2.4,2.2)\footnotesize
\pspolygon[fillstyle=solid,fillcolor=lightgray](0.75,0)(1.75,0.5)(2,1.4)(1.5,2)(0.25,1.65)(0,0.75)
\psline(0.75,0)(1.5,2) \psline(1.75,0.5)(0.25,1.65)
\psline(2,1.4)(0,0.75)
\psdots[dotsize=4pt](0.75,0)(1.75,0.5)(2,1.4)(1.5,2)(0.25,1.65)(0,0.75)
\uput[l](0.75, 0){$\mu(o)$} \uput[r](1.5,2){$\mu(r)$}
\end{pspicture}}

\subfigure[]{
\begin{pspicture}(-0.4,-0.2)(2.4,2.2)\footnotesize
\pspolygon[fillstyle=solid,fillcolor=lightgray](1,0)(1.9,0.5)(2,1.25)(1.5,2)(0,1.3)(0,0.5)
\psline(1,0)(1.1,0.5)(2,1.25)(1.1,0.5)(0,1.3)
\psline(1.9,0.5)(1.25,1.25)(1.5,2)(1.25,1.25)(0,0.5)
\psdots[dotsize=4pt](1,0)(1.9,0.5)(2,1.25)(1.5,2)(0,1.3)(0,0.5)(1.1,0.5)(1.25,1.25)
\uput[dl](1, 0){$\mu(o)$} \uput[r](1.5,2){$\mu(r)$}
\end{pspicture}}

}
\end{center}
\caption{\label{figure: image of moment} Examples of
eight types of possible index increasing GKM graphs}
\end{figure}
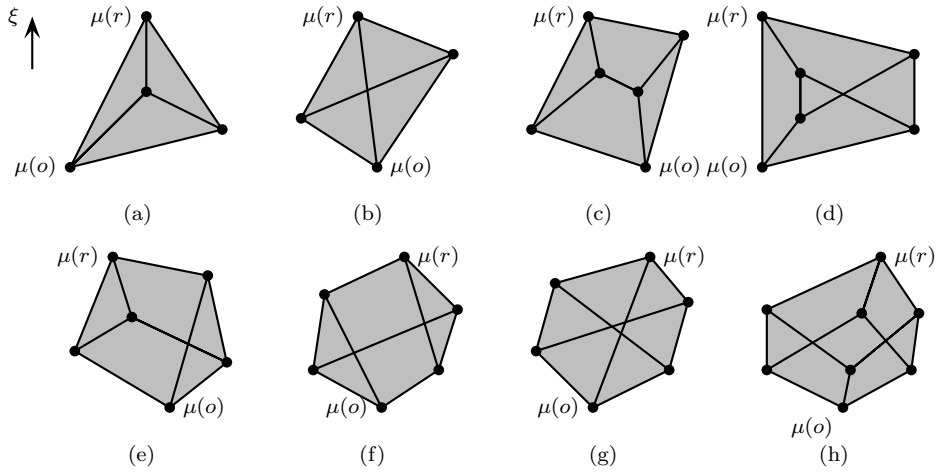

The proof of Proposition \ref{proposition: graph
shape} will be given in Section
\ref{secProofOfPropositionRefPropositionGraphShapeAndRefPropositionConvexIf8}.
In Figure \ref{figure: image of moment}, examples of the eight types of GKM graphs in Table \ref{table: graph shape}
are illustrated. Note that Proposition \ref{proposition: graph shape} does not claim that every possible 
index increasing GKM graph is one of those in Figure \ref{figure: image of moment}.  
For example, there exists an index increasing GKM graph satisfying Table \ref{table: graph shape}.(h) 
but is different from Figure \ref{figure: image of moment}.(h), see Figure \ref{figure_h}. 
Thus we may call Proposition
\ref{proposition: graph shape} a {\em weak classification} of index increasing GKM graphs of
closed six-dimensional Hamiltonian GKM manifolds.
Nevertheless, Table \ref{table: graph
shape}.(a)$\sim$(g) ($\mathcal{V} \leq 6$) are
corresponding to the Morton's classification of index increasing
GKM graphs of closed six-dimensional Hamiltonian GKM
manifolds with vertices less than or equal to six, see
\cite{Mo}.

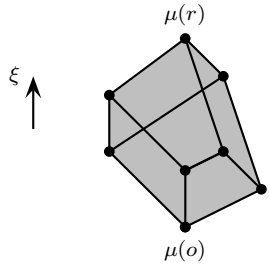
\begin{figure}[ht]
\begin{center}

\begin{pspicture}(-1,0)(2,3)\footnotesize
\pspolygon[fillstyle=solid,fillcolor=lightgray](1,0)(2,0.5)(1.5,2)(1,2.5)(0,1.75)(0,1)(1,0)

\psline(0,1.75)(1,0.75)(1.5,1)(1,2.5)
\psline(1,0)(1,0.75)(1.5,1)(2,0.5)

\psline(0,1)(1.5,2)

\psdots[dotsize=4pt](1,0)(2,0.5)(1.5,2)(1,2.5)(0,1.75)(0,1)(1,0)(1,0.75)(1.5,1)

\uput[d](1, 0){$\mu(o)$} \uput[u](1,2.5){$\mu(r)$}
\psline[arrowsize=6pt]{->}(-1,1.3)(-1,2)
\uput[l](-1,2){$\xi$}
\end{pspicture}

\end{center}
\caption{\label{figure_h} Example of Table
\ref{table: graph shape}.(h)}
\end{figure}

\begin{remark}
We can easily check that Tolman's example (Example \ref{example_tolman})
corresponds to Table \ref{table: graph shape}.(d).
See also \cite[Example 5.2 and Figure 1]{GT}.
\end{remark}

The following proposition will be used to prove our main theorem in Section \ref{ssecProofOfTheoremRefTheoremmain},
where the proof will be given in Section  \ref{secProofOfPropositionRefPropositionGraphShapeAndRefPropositionConvexIf8}. 

\begin{proposition} \label{proposition: convex if 8}
Suppose that $\Gamma$ is of type Table \ref{table: graph shape}.(h). 
Then every tetragonal ascending cycle in $\Gamma$ is convex. In particular, $\Theta(p,q)$ is positive 
for every index-two vertex $p$ and index-four vertex $q$ of $\Gamma$.
\end{proposition}

\subsection{Proof of Theorem \ref{theorem_main}}
\label{ssecProofOfTheoremRefTheoremmain}

Now, we are ready to prove our main theorem.

\begin{proof}[Proof of Proposition \ref{proposition_main}](Proof of Theorem \ref{theorem_main})
We first consider the case where a GKM graph $\Gamma$
satisfies Table \ref{table: graph shape}.(a) or (b).
In this case, we have $b_2 = 1$ and $H^2(M)$ is generated by the
symplectic class $[\omega]$. Since $[\omega^2] \neq 0$
in $H^4(M)$, the hard Lefschetz property of $(M,\omega)$ is automatically
satisfied. 

Second, suppose that $\Gamma$ satisfies Table
\ref{table: graph shape}.(c), (e), or (f).
In this case, we have $b_2=2$ and there is only one tetragonal ascending cycle starting at an index-two vertex. 
(See Table \ref{table: graph shape}.)
This implies that  the number of non-oriented edges connecting an index-two vertex and a four vertex is three. 
Since $A_2(M,\omega)$ is a $2 \times 2$ matrix with exactly three nonzero entries by Proposition
\ref{proposition_matrix_coefficient}, $A_2(M,\omega)$ should be nonsingular.

Third, assume that $\Gamma$ satisfies Table \ref{table: graph shape}.(h). Then $\Gamma$ has three
index-two vertices so that $b_2=3$. In particular, $A_2(M,\omega)$ is a $3 \times 3$ matrix.
Also, there are three tetragonal ascending cycles starting at an index-two vertex, that is, the ascending cycle 
starting at an index-two vertex $p_k$ is tetragonal for every $k=1,2,3$. 
Moreover, the ascending cycles are convex by Proposition \ref{proposition: convex if 8}.
Thus if $p_k$ and $q_j$ are adjacent, then $\Theta (p_k, q_j)$ is positive by Proposition
\ref{proposition_positive_if_convex}, and therefore $a_{jk}$ is positive for $1 \leq j, k \leq 3$ 
with $(p_k, q_j) \in E_\Gamma$ by Proposition
\ref{proposition_matrix_coefficient}.
So, by Proposition \ref{proposition_matrix_coefficient}, there are exactly three zeros in $A_2(M,\omega)$ and 
each zero appears exactly one time in each row and column. 
Reordering $p_k$'s and $q_j$'s, if necessary, we may assume that 
the diagonal entries of $A_2 (M,\omega)$ are all zero. Then,
\[
\det \, A_2(M,\omega) = a_{12} a_{23} a_{31} + a_{13}
a_{21} a_{32} > 0.
\]
Therefore, $A_2(M,\omega)$ is nonsingular.

Finally, consider the case where $\Gamma$ satisfies Table \ref{table: graph shape}.(d) or (g). 
In this case, we have $b_2=2$ (and hence $A_2(M,\omega)$ is a $2 \times 2$ matrix) 
and there are two tetragonal ascending cycles starting at an index-two vertex. 
In other words, the ascending cycle starting at each index-two vertex $p_k$ is tetragonal and hence 
it is adjacent to both $q_1$ and $q_2$. Thus all entries of $A_2(M,\omega)$ are nonzero by
Proposition \ref{proposition_matrix_coefficient}. 

To show that the determinant of $A_2(M,\omega)$ is
nonzero, we apply column operation on $A_2(M,\omega)$ to obtain a triangular matrix. 
To do this, we need the following lemma.

\begin{lemma} \label{lemma: two are independent}
Let $t_1$ and $t_2$ be two arbitrary nonzero real
numbers. If $\Gamma$ satisfies Table \ref{table: graph
shape}.(d) or (g), then the following (degree four) graph cohomology
class
\begin{equation}
\label{equation: linear combination} t_1 \cdot
\tau_{p_1}^+ \cdot \big( [\widetilde{\omega_\mu}] +
\mu(p_1) \big) + t_2 \cdot \tau_{p_2}^+ \cdot
\big( [\widetilde{\omega_\mu}] + \mu(p_2) \big)
\end{equation}
does not vanish simultaneously on $q_1$ and $q_2.$
\end{lemma}

\begin{proof}
Note that the class (\ref{equation: linear
combination}) vanishes on $o,$ $p_1,$ $p_2$ by Theorem
\ref{theorem_Thom_class} and Lemma
\ref{lemma_equivariant_symplectic}. Moreover, if
(\ref{equation: linear combination}) vanishes on $q_1$
and $q_2$ simultaneously, then (\ref{equation: linear
combination}) should be the zero class by Lemma
\ref{lemma_zero_at_a_vertex}.(2). Thus we need only show that (\ref{equation: linear
combination}) never vanishes on $r$.

Therefore, it is enough to prove that the following two linear polynomials 
\begin{equation}
\label{equation: two are independent} %
\Big[ \tau_{p_1}^+ \cdot \big( [\widetilde{\omega_\mu}]
+ \mu(p_1) \big) \Big] (r) \quad \text{ and }
\quad \Big[ \tau_{p_2}^+ \cdot \big(
[\widetilde{\omega_\mu}] + \mu(p_2) \big) \Big] (r)
\end{equation}
in $\sym(\mathfrak{t}^*)$ are $\R$-linearly independent. 
We first compute
$\tau_{p_k}^+(r)$ as follows. Since $\tau_{p_k}^+$ is
zero at $o$ for $k=1,2$ by Theorem \ref{theorem_Thom_class}, and
$o$ and $r$ are adjacent by Table \ref{table: graph shape}.(d) and (g), we have
\[
\tau_{p_k}^+(r) = d_k \cdot \alpha (r, o)
\]
for some real numbers $d_k$ by Lemma \ref{lemma_zero_at_a_vertex}.(1).

We claim that $d_k$'s are all nonzero. Suppose that
$d_k$ is zero for some $k,$ i.e. $\tau_{p_k}^+(r) =
0.$ Without loss of generality, we may assume that
$k=1$. Then $\tau_{p_1}^+$ vanishes on $r.$ Moreover,
$\tau_{p_1}^+$ vanishes on $p_{2}$ by (\ref{equation:
supp of Thom class}). Since each of $q_1$ and $q_2$ is
adjacent to both $p_1$ and $p_2$,
$\tau_{p_1}^+(q_j)$ is divided by both $\alpha (q_j,
r)$ and $\alpha (q_j, p_{2})$ for each $j=1,2$ by
Lemma \ref{lemma_zero_at_a_vertex}.(1). However, two
weights $\alpha (q_j, r)$ and $\alpha (q_j, p_{2})$
are linearly independent by the GKM condition (2) and $\tau_{p_1}^+(q_j)$ is of polynomial
degree 1 in $\sym (\mathfrak{t}^*).$ Thus we have
$\tau_{p_1}^+(q_j)=0$ and it is a contradiction by
Lemme \ref{lemma_condition_for_nonzero}. 
Thus $d_1$ is nonzero. Also, we obtain $d_2 \neq 0$ in a similar way. 
Therefore, the polynomials in (\ref{equation: two are
independent}) can be expressed by
\[
\big[ \tau_{p_k}^+ \cdot \big( [\widetilde{\omega_\mu}] +
\mu(p_k) \big) \big] (r) = d_k \cdot \alpha (r, o)
\cdot \big( - \mu(r) + \mu(p_k) \big)
\]
by Lemma \ref{lemma_equivariant_symplectic} for $k=1,2$.

Now, it is enough to show that
\[
\mu(r) - \mu(p_1) \qquad \text{ and } \qquad \mu(r) -
\mu(p_2)
\]
are $\R$-linearly independent. 
To the contrary, suppose that they are linearly dependent. 
Then $\mu(r)$, $\mu(p_1)$, and $\mu(p_2)$ should be collinear. Let us first consider the 
case of Table \ref{table: graph shape}.(d). Then there
exists index-four interior vertex, which is assumed to
be $q_1$, adjacent to $r$, $p_1$, and $p_2$. 
Similarly, we can easily see that $r$ is adjacent to
$o$, $q_1$, and $q_2$. Note that if $\mu(p_1)$ and
$\mu(p_2)$ are in the same side with respect to the
straight line $\overleftrightarrow{\mu(r)\mu(q_1)}$,
then both
$\mu(o)$ and $\mu(q_2)$ must be in the same side with
$\mu(p_k)$'s by Lemma \ref{lemma_GKM_weight_pair_2} so
that $\overleftrightarrow{\mu(r)\mu(q_1)}$ is on the
boundary of $\mu(M)$, which contradicts that $q_1$ is
an interior point of the moment polytope $\mu(M)$. Thus $\mu(p_1)$ and $\mu(p_2)$
cannot be in the same side with respect to the
straight line $\overleftrightarrow{\mu(r)\mu(q)},$ and
hence $\mu(r)$, $\mu(p_1)$, and $\mu(p_2)$ are not
colinear. 

For the case of (g), $\mu(p_1)$ and $\mu(p_2)$ are vertices of the moment polytope $\mu(M)$ as well as $\mu(r)$. 
Then it is straightforward that $\mu(r) - \mu(p_1)$ and $\mu(r) - \mu(p_2)$ are linearly independent.
\end{proof}

We go back to the proof of Proposition \ref{proposition_main}.
Since every $a_{jk}$ is nonzero, we can take $t_0 = - \frac{a_{12}}{a_{11}} \neq 0$ so
that $a_{12} + t_0 \cdot a_{11}=0$. Since
\begin{equation*}
\det \left(
\begin{array}{cc}
a_{11} & a_{12}  \\
a_{21} & a_{22}  \\
\end{array}
\right)
=
\det \left(
\begin{array}{cc}
a_{11} & a_{12} + t_0 \cdot a_{11} \\
a_{21} & a_{22} + t_0 \cdot a_{21} \\
\end{array}
\right)
=
\det \left(
\begin{array}{cc}
a_{11} & 0  \\
a_{21} & a_{22} + t_0 \cdot a_{21} \\
\end{array}
\right),
\end{equation*}
it is enough to show that $a_{22} + t_0 \cdot a_{21} \ne 0.$ 

Consider the
following equivariant cohomology class
\[
\tau_{p_k}^+ \cdot \big( [\widetilde{\omega_\mu}]
+ \mu(p_k) \big) \cdot \tau_{q_j}^- \in H^6_T(M).
\]
Using $f^* \big( [\widetilde{\omega_\mu}] - \mu(p_k) \big) = [\omega]$
and the ABBV-localization
theorem, we have
\begin{align}
\label{equation: coefficient calculation}
a_{jk} &= \langle f^* (\tau_{p_k}^+) \wedge [\omega] \wedge f^* (\tau_{q_j}^-), [M] \rangle \\
\notag &= \int_M ~ \tau_{p_k}^+ \cdot \big(
[\widetilde{\omega}_\mu] + \mu(p_k) \big) \cdot
\tau_{q_j}^- \\
\notag&= \sum_{v \in V_\Gamma} \Big[
\tau_{p_k}^+ \cdot \big( [\widetilde{\omega}_\mu]
+ \mu(p_k) \big) \cdot
\tau_{q_j}^- \Big] (v) \Big/ \Lambda_v \\
\notag &= \Big[ \tau_{p_k}^+ \cdot \big(
[\widetilde{\omega}_\mu] + \mu(p_k) \big) \cdot
\tau_{q_j}^- \Big] (q_j) \Big/ \Lambda_{q_j} \\
\notag &= \tau_{p_k}^+ (q_j) \cdot \big(
-\mu(q_j) + \mu(p_k) \big) \Big/ \Lambda_{q_j}^+
\qquad \quad \Big( ~ = \Big[ \tau_{p_k}^+
\cdot \big( [\widetilde{\omega}_\mu] + \mu(p_k)
\big) \Big] (q_j) \Big/ \Lambda_{q_j}^+ ~ \Big).
\end{align}
In the fourth equality, we use Theorem
\ref{theorem_Thom_class} and the followings :
\begin{align*}
\bullet ~&\supp \tau_{p_k}^+ \subset \{ p_k \} \cup \{
\text{index-4 vertices adjacent to } p_k \}
\cup \{ \text{the index-6 vertex } r \}, \\
\bullet ~& \supp \tau_{q_j}^- \subset \{ q_j \} \cup
\{ \text{index-2 vertices adjacent to } q_j \}
\cup \{ \text{the index-0 vertex } o \}, ~\mathrm{and}\\
\bullet ~& \big( [\widetilde{\omega_\mu}] + \mu(p_k) \big)
(p_k) = -\mu(p_k) + \mu(p_k) = 0
\end{align*}
obtained from (\ref{equation: supp of Thom class}) and
Lemma \ref{lemma_equivariant_symplectic}. 
Then, by \eqref{equation: coefficient calculation}, we can easily see that 
\begin{align*}
a_{jk} ~&=~ \Big[ \tau_{p_k}^+ \cdot \big(
[\widetilde{\omega}_\mu] + \mu(p_k) \big) \Big] (q_j)
\Big/ \Lambda_{q_j}^+,
\qquad \text{and} \\
a_{j2} + t_0 \cdot a_{j1} ~&=~ \Big[ \tau_{p_2}^+
\cdot \big( [\widetilde{\omega}_\mu] + \mu(p_2) \big) +
t_0 \cdot \tau_{p_1}^+ \cdot \big(
[\widetilde{\omega}_\mu] + \mu(p_1) \big) \Big] (q_j)
\Big/ \Lambda_{q_j}^+.
\end{align*}
Since $t _0\neq 0$, both $a_{12} + t_0 \cdot a_{11}$ and $a_{22} + t_0 \cdot a_{21}$ do not vanish simultaneously by Lemma
\ref{lemma: two are independent}. Therefore, we have $a_{22} + t_0 \cdot a_{21} \ne 0$.
This completes the proof.
\end{proof}

\bigskip

\section{Proof of Proposition \ref{proposition: graph shape} and \ref{proposition: convex if 8}}
\label{secProofOfPropositionRefPropositionGraphShapeAndRefPropositionConvexIf8}

In this section, we prove Proposition
\ref{proposition: graph shape} and \ref{proposition:
convex if 8} used in Section \ref{secSixDimensionalHamiltonianGKMManifoldsWithAnIndexIncreasingGraph}.
To begin with, we introduce the following terminologies.

\begin{definition} \label{definition: interior and
boundary} A vertex $v$ is called a {\em boundary
vertex} (resp. an {\em interior vertex}) if $\mu(v)$ is
contained in the boundary (resp. interior) of $\mu(M)$.
Similarly, an edge $e$ is called a {\em boundary edge} (resp. an {\em interior edge}) if
$\mu(S_e^2)$ is contained in the boundary (resp. interior) of $\mu(M)$.
A path $(v_0, \cdots, v_l)$ of $\Gamma$ is called a {\em boundary path} 
if each edge $(v_j, v_{j+1})$ is a boundary for every $0 \le j \le l-1$. 
\end{definition}

Now, we give the proof of Proposition
\ref{proposition: graph shape} as follows.

\begin{proof}[Proof of Proposition \ref{proposition: graph shape}]
Consider two ascending boundary paths
\[
(v_0, \cdots, v_l) \qquad \text{and} \qquad
(v_0^\prime, \cdots, v_{l^\prime}^\prime)
\]
from $o$ to $r$. Then $\mu(M)$ is a convex $(l+l^\prime)$-gon by the 
Atiyah-Guillemin-Sternberg convexity theorem \cite{At,GS}
so that both paths
cannot have length one simultaneously. Moreover, by
the index increasing property, the lengths of the two paths
are less than or equal to three, i.e. $l, \, l^\prime
\le 3.$ Therefore, we have
\[
2 \le l \cdot l^\prime \quad \text{ and } \quad l, \,
l^\prime \le 3.
\]
Without loss of generality, we may assume that $l \le l^\prime.$ Then there are exactly five possible cases : 
\[
	(l, l') \in \{ (1,2), (2,2), (1,3), (2,3), (3,3) \}.
\]

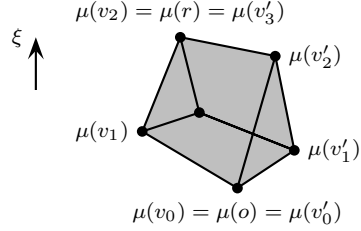
\begin{figure}[ht]
\begin{center}
\begin{pspicture}(0,0)(2,2.5)\footnotesize
\pspolygon[fillstyle=solid,fillcolor=lightgray](0.5,2)(1.75,1.75)(2,0.5)(1.25,0)(0,0.75)
\psline(0.5,2)(0.75,1)(2,0.5)(0.75,1)(0,0.75)
\psline(1.75,1.75)(1.25,0)
\psdots[dotsize=4pt](0.5,2)(0.75,1)(1.75,1.75)(2,0.5)(1.25,0)(0,0.75)

\uput[l](0,0.75){$\mu(v_1)$} \uput[r](2,0.5){$\mu(v_1^\prime)$}
\uput[r](1.75,1.75){$\mu(v_2^\prime)$} \uput[u](0.5, 2){$\mu(v_2) = \mu(r) = \mu(v_3^\prime)$} 
\uput[d](1.25,0){$\mu(v_0) = \mu(o) = \mu(v_0^\prime)$}

\psline[arrowsize=6pt]{->}(-1.4,1.3)(-1.4,2)
\uput[l](-1.4,2){$\xi$}
\end{pspicture}
\end{center}
\caption{\label{figure: ascending boundary path} Example of ascending boundary paths for $(l,l') = (2,3)$}
\end{figure}

If $l=1, l^\prime=2,$ then $\mu(M)$ is a triangle so that $o$ and $r$ are adjacent.
We may assume that $v_1^\prime$ is of index-two. (If not, then $v_1^\prime$ is of index-two with respect to $-\xi$ 
so that we need only take $-\xi$ instead of $\xi$.) 
Then there exists at least one index-four interior vertex by the
Poincar\'{e} duality.
Moreover, there cannot exist more than one index-four
vertex by the three valency at $r$, since any index-four vertex is adjacent to $r$ 
by the index-increasing property and $r$ is already
adjacent to two vertices $o$ and $v_1^\prime$. Therefore, $\Gamma$ has 
only one index-four vertex so that $\Gamma$ has four vertices, that is, $\Gamma$ 
is a complete graph and the unique ascending cycle
$\gamma_{v_1^\prime}$ starting at the unique index-two vertex $v_1^\prime$ is triangular. 
Thus $\Gamma$ is the case of Table \ref{table: graph shape}.(a).

If $l=2,$ $l^\prime=2,$ then $\mu(M)$ is a tetragon.
Then each of $o$ and $r$ is adjacent to
the boundary vertices $v_1$ and $v_1^\prime$. We first
show that $v_1$ and $v_1^\prime$ have different
indices. If $v_1$ and $v_1^\prime$ have the same
index, say two, then there should be at least two
index-four interior vertices by the Poincar\'{e}
duality. Thus $r$ should be adjacent to at least four
vertices and this contradicts the three valency at $r$.
Therefore, $v_1$ and $v_1^\prime$ must have different
indices. Assume that $v_1$ is of index-two and $v_1^\prime$ is of index-four. 
Then $\gamma_{v_1}$ is triangular since $v_1$ is adjacent to $r$. 

Now, there are two possible cases according to adjacency of $o$ and $r$.
If $o$ and $r$ are adjacent, then $o$ (resp. $r$) is adjacent to the three vertices  
$r$, $v_1$, and $v_1^\prime$ (resp. $o$, $v_1$, and $v_1^\prime$)
so that there is no other vertex except for $o,$ $r,$ $v_1,$ and $v_1^\prime$ 
since any vertex other than $o,$ $r$ should be adjacent to $o$ or $r$ by the index
increasing property. In other words, $\Gamma$ has four
vertices and $v_1$ must adjacent to $v_1^\prime$. This is the case of Table \ref{table: graph shape}.(b).

Second, assume that $o$ and $r$ are not adjacent.
Since each of $o$ and $r$ is adjacent to $v_1$ and $v_1^\prime$, 
there are exactly two interior vertices by the three valency of $\Gamma$ and the Poincar\'{e} duality, 
and therefore $\Gamma$ has six vertices. 
Since $v_1$ and $v_1^\prime$ have different indices, two interior
vertices have different indices by the Poincar\'{e}
duality. Let $p$ and $q$ be the index-two and index-four interior vertex respectively.
Note that $v_1$ and $v_1^\prime$ are not adjacent, otherwise 
$\Gamma$ cannot be three-valent at $p$ and $q$.
Therefore, $p$ and $v_1^\prime$ are adjacent and that $v_1$ and $q$ are adjacent by the index
increasing property. Also $p$ is adjacent to $q$ by the three valency of $\Gamma$. 
Consequently, $p$ is adjacent to three vertices $q,$ $v_1^\prime,$ $o$ so that $\gamma_p$ is
tetragonal. This is the case of Table \ref{table: graph shape}.(c).

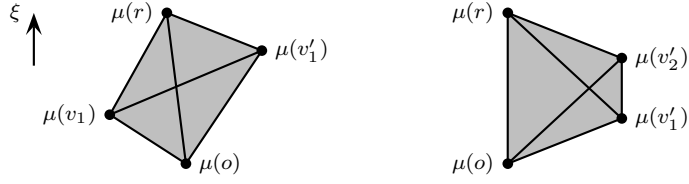
\begin{figure}[ht]
\begin{center}
\mbox{

\subfigure{
\begin{pspicture}(-1.5,0)(3.5,2)\footnotesize
\pspolygon[fillstyle=solid,fillcolor=lightgray](1,0)(0,0.65)(0.75,2)(2,1.5)
\psline(1,0)(0.75,2) \psline(0,0.65)(2,1.5)
\psdots[dotsize=4pt](1,0)(0,0.65)(0.75,2)(2,1.5)
\uput[r](1, 0){$\mu(o)$} \uput[l](0.75,2){$\mu(r)$}
\uput[l](0,0.65){$\mu(v_1)$} \uput[r](2,1.5){$\mu(v_1^\prime)$}

\psline[arrowsize=6pt]{->}(-1,1.3)(-1,2)
\uput[l](-1,2){$\xi$}
\end{pspicture}}

\subfigure{
\begin{pspicture}(-1.5,0)(3.5,2)\footnotesize
\pspolygon[fillstyle=solid,fillcolor=lightgray](0,0)(1.5,0.6)(1.5,1.4)(0,2)
\psline(0,0)(1.5,1.4) \psline(1.5,0.6)(0,2)
\psdots[dotsize=4pt](0,0)(1.5,0.6)(1.5,1.4)(0,2) %
\uput[l](0, 0){$\mu(o)$} \uput[l](0,2){$\mu(r)$}
\uput[r](1.5,0.6){$\mu(v_1^\prime)$}
\uput[r](1.5,1.4){$\mu(v_2^\prime)$}
\end{pspicture}}
}
\end{center}
\caption{\label{figure: (b) case} Examples of Table
\ref{table: graph shape}.(b)}
\end{figure}

If $l=1, l^\prime=3,$ then $\mu(M)$ is a tetragon and $o$ is adjacent to $r$. 
Note that $v_1^\prime$ and $v_2^\prime$ are of index-two and of index-four by the index increasing property, respectively. 
Note that $o$ (resp. $r$) is adjacent to $r$ and $v_1^\prime$ (resp. $o$ and $v_2^\prime$) and hence 
there are at most two interior vertices. By the Poincar\'{e} duality, the number of interior vertices is zero or two. 

First, if there is no interior vertex, then $\Gamma$ has four vertices and $v_1^\prime$ (resp. $v_2^\prime$) 
is adjacent to $r$ (resp. $o$), which is the case of Table \ref{table: graph shape}.(b). 
Second, assume that there are two interior vertices,
namely, the index-two interior vertex $p$ and the
index-four interior vertex $q.$ Then $r$ is adjacent
to three vertices $o,$ $v_2^\prime,$ and $q.$
Similarly, $o$ is adjacent to $r$, $v_1^\prime$, and
$p$. By the three valency of $\Gamma$, each of $p$ and $v_1^\prime$ is adjacent to $v_2^\prime$ and $q$, 
and therefore the ascending cycles $\gamma_p$ and $\gamma_{v_1^\prime}$ are tetragonal. 
This is the case of Table \ref{table: graph shape}.(d). 

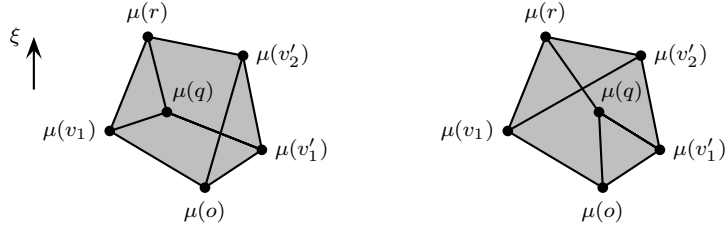
\begin{figure}[ht]
\begin{center}
\mbox{

\subfigure{
\begin{pspicture}(-1.5,0)(3.5,2)\footnotesize
\pspolygon[fillstyle=solid,fillcolor=lightgray](0.5,2)(1.75,1.75)(2,0.5)(1.25,0)(0,0.75)
\psline(0.5,2)(0.75,1)(2,0.5)(0.75,1)(0,0.75)
\psline(1.75,1.75)(1.25,0)
\psdots[dotsize=4pt](0.5,2)(0.75,1)(1.75,1.75)(2,0.5)(1.25,0)(0,0.75)

\uput[l](0,0.75){$\mu(v_1)$} \uput[r](2,0.5){$\mu(v_1^\prime)$}
\uput[r](1.75,1.75){$\mu(v_2^\prime)$} \uput[u](0.5,
2){$\mu(r)$} \uput[d](1.25,0){$\mu(o)$} \uput[ur](0.7,1){$\mu(q)$}

\psline[arrowsize=6pt]{->}(-1,1.3)(-1,2)
\uput[l](-1,2){$\xi$}
\end{pspicture}}

\subfigure{
\begin{pspicture}(-1.5,0)(3.5,2)\footnotesize
\pspolygon[fillstyle=solid,fillcolor=lightgray](0.5,2)
(1.75,1.75)(2,0.5)(1.25,0)(0,0.75)
\psline(0.5,2)(1.2,1)(2,0.5)(1.2,1)
\psline(0,0.75)(1.75,1.75) \psline(1.2,1)(1.25,0)
\psdots[dotsize=4pt](0.5,2)(1.2,1)(1.75,1.75)(2,0.5)
(1.25,0)(0,0.75)

\uput[l](0,0.75){$\mu(v_1)$} \uput[r](2,0.5){$\mu(v_1^\prime)$}
\uput[r](1.75,1.75){$\mu(v_2^\prime)$} \uput[u](0.5,
2){$\mu(r)$} \uput[d](1.25,0){$\mu(o)$} \uput[ur](1.1,1){$\mu(q)$}

\end{pspicture}}

}
\end{center}
\caption{\label{figure: 2,3 case} The case of $l = 2$
and $l^\prime = 3$ }
\end{figure}

If $l=2, l^\prime=3,$ then $\mu(M)$ is a pentagon.
Since $(v_0^\prime, v_1^\prime, v_2^\prime, v_3^\prime)$ is
an ascending boundary path from $v_0^\prime = o$ to
$v_3^\prime = r,$ $v_1^\prime$ is of index-two and 
$v_2^\prime$ is of index-four. Furthermore, we may assume that $v_1$ is of index two. 
Since $v_1$ is adjacent to $r,$ the ascending
cycle $\gamma_{v_1}$ is triangular. Note that there is exactly one interior vertex, say $q$, by the three valency of $\Gamma$
and the Poincar\'{e} duality. 
Then $r$ is adjacent to $v_1,$ $v_2^\prime,$ and $q$ so that $r$ is not adjacent to $o$. 
Also, $r$ is not adjacent to $v_1^\prime$. Thus $v_1^\prime$ is adjacent to $q$ and $v_2^\prime$
so that the ascending cycle
$\gamma_{v_1^\prime}$ is tetragonal. Consequently, there
is only one tetragonal ascending cycle $\gamma_{v_1^\prime}$ 
and this is the case Table \ref{table: graph shape}.(e).

If $l=3, l^\prime=3,$ then $\mu(M)$ is a hexagon. By
index increasing property, $v_1$ and $v_1^\prime$ are
of index-two, and $v_2$ and $v_2^\prime$ are of index
four. By the three valency at $o$ and $r$, there exist
at most two interior vertices so that there are two
possible cases according to the number of interior vertices. 
If $\Gamma$ has six vertices (with no interior vertex),
then this is the case of Table \ref{table: graph shape}.(f) if $o$ and $r$ is not adjacent, and 
of Table \ref{table: graph shape}.(g) if $o$ and $r$ is adjacent.
Also, if $\Gamma$ has eight vertices (with two interior vertices), 
then $r$ should be adjacent to three index-four vertices so that $r$ is not adjacent to $o$. 
Thus any ascending cycle is tetragonal and this is the case of Table \ref{table: graph shape}.(h)
\end{proof}

Now, we prove Proposition \ref{proposition: convex if 8}.
We first recall the following.

\begin{lemma} \label{lemma: interior}
A vertex $v$ is an interior vertex if and only if $\sum_{1 \le j \le
3} ~ \R^+ \cdot \alpha_{j, v} = \mathfrak{t}^*.$ In particular, if
$v$ is an interior vertex, then $\alpha_{1, v},$ $\alpha_{2, v}$ are
not in the same side with respect to $\R \cdot \alpha_{3, v}.$
\end{lemma}

\begin{proof}
See \cite[Lemma 2 and Example 2]{Km}.
\end{proof}

\begin{proof}[Proof of Proposition
\ref{proposition: convex if 8}] We label each vertex as
in Figure \ref{figure: proof of convex} :
\begin{itemize}
    \item two ascending boundary paths from $o$ to $r$ are $(o, p_1, q_1, r)$ and $(o, p_3, q_3, r)$, and
    \item $p_2$ and $q_2$ are the interior vertices of index-two and four, respectively.
\end{itemize}
Note that 
\begin{itemize}
	\item every ascending cycle starting at an index-two vertex is tetragonal by Table \ref{table: graph shape},
		and therefore each $p_k$ (resp. $q_j$) is not adjacent to $r$ (resp. $o$), and 
	\item each tetragonal ascending cycle contains two index-four vertices, 
		it contains at least one boundary vertex of index-four.
\end{itemize}
We also note that, by interchanging $p_1$ and $p_3$ (resp. $q_1$ and $q_3$) if  necessary, 
there are exactly four types of ascending cycles in $\Gamma$ : 
(i) an ascending cycle $\gamma_p$ ($p$ is any index-two vertex) contains $q_1$ and $q_3,$ (ii)
$\gamma_{p_1}$ contains $q_1$ and $q_2$, (iii) $\gamma_{p_2}$ contains $q_1$ and $q_2$, and (iv) $\gamma_{p_3}$ contains $q_1$ and $q_2$.

\subsection*{Case (i): $\gamma_p$ contains $q_1$ and $q_3$}
Note that $(q_1, r)$ and $(q_3, r)$ of $\gamma_p$ are boundary edges so that $\gamma_p$
cannot be crossed by Lemma \ref{lemma: tetragon}, see
Figure \ref{figure: proof of convex}.(a) for example. Furthermore, 
$p$ is below $q_1$, $q_3$, and $r$ by the index increasing property so that  
$p$ is not contained in the interior of $\Conv \{p, q_1, q_3, r\}.$
Thus $\gamma_p$ is convex.

\begin{figure}[ht]
\begin{center}
\mbox{\subfigure[]{
\begin{pspicture}(0,-0.2)(3,3)\footnotesize
\pspolygon[fillstyle=solid,fillcolor=lightgray,
linestyle=none](1,0)(2.5,0.5)(3,1.5)(2,3)(0.25,2.25)(0,1)
\psline(2,3)(3,1.5)(1.5,0.75)(0.25,2.25)(2,3)
\psdots[dotsize=3pt](1,0)(2.5,0.5)(3,1.5)(2,3)(0.25,2.25)(0,1)(1.5,0.75)(1.5,1.75)
\uput[d](1,0){$\mu(o)$} \uput[dr](2.5,0.5){$\mu(p_1)$}
\uput[r](3,1.5){$\mu(q_1)$} \uput[u](2,3){$\mu(r)$}
\uput[ul](0.25,2.25){$\mu(q_3)$} \uput[dl](0,1){$\mu(p_3)$}
\uput[d](1.5,0.75){$\mu(p_2)$} \uput[d](1.5,1.75){$\mu(q_2)$}
\psline[arrowsize=6pt]{->}(-1,1.3)(-1,2.2)
\uput[r](-1,2.2){$\xi$}
\end{pspicture}}

\qquad

\subfigure[]{
\begin{pspicture}(0,-0.2)(3,3)\footnotesize
\pspolygon[fillstyle=solid,fillcolor=lightgray,
linestyle=none](1,0)(2,0.3)(3,1.5)(1.3,3)(0.25,2.25)(0,1)
\psdots[dotsize=3pt](1,0)(2,0.3)(3,1.5)(1.3,3)(0.25,2.25)(0,1)(2.5,1.4)(2,1.65)
\psline(1.3,3)(3,1.5)(2,0.3)(2,1.65)(1.3,3)
\uput[d](1,0){$\mu(o)$} \uput[dr](2,0.3){$\mu(p_1)$}
\uput[r](3,1.5){$\mu(q_1)$} \uput[u](1.3,3){$\mu(r)$}
\uput[ul](0.25,2.25){$\mu(q_3)$} \uput[dl](0,1){$\mu(p_3)$}
\uput[d](2.5,1.4){$\mu(p_2)$} %
\uput[l](2,1.65){$\mu(q_2)$} \psline[linestyle=dotted,
dotsep=1.5pt,linewidth=1pt](2,1.65)(2.5,1.4)
\psline[linestyle=dotted,
dotsep=1.5pt](2.5,1.4)(3,1.5)
\psline[linestyle=dotted,
dotsep=1.5pt,linewidth=1pt](1,0)(2.5,1.4)
\end{pspicture}}

\qquad

\subfigure[]{
\begin{pspicture}(0,-0.2)(3,3)\footnotesize
\pspolygon[fillstyle=solid,fillcolor=lightgray,
linestyle=none](1,0)(2.5,0.5)(3,1.5)(2,3)(0.25,2.25)(0,1)
\psline[linestyle=dotted,
dotsep=1.5pt,linewidth=1pt](3,1.5)(2.5,0.5)(1,0)(0,1)(0.25,2.25)(2,3)
\psline[linestyle=dotted,
dotsep=1.5pt,linewidth=1pt](2.25,1.25)(1,0)
\psdots[dotsize=3pt](1,0)(2.5,0.5)(3,1.5)(2,3)(0.25,2.25)(0,1)(2.25,1.25)(1.75,2)
\psline(2,3)(1.75,2)(2.25,1.25)(3,1.5)(2,3)
\uput[d](1,0){$\mu(o)$} \uput[dr](2.5,0.5){$\mu(p_1)$}
\uput[r](3,1.5){$\mu(q_1)$} \uput[u](2,3){$\mu(r)$}
\uput[ul](0.25,2.25){$\mu(q_3)$} \uput[dl](0,1){$\mu(p_3)$}
\uput[dr](2.25,1.25){$\mu(p_2)$} \uput[ul](1.75,2){$\mu(q_2)$}
\end{pspicture}}
}

\mbox{

\subfigure[]{
\begin{pspicture}(0,-0.2)(3,3)\footnotesize
\pspolygon[fillstyle=solid,fillcolor=lightgray,
linestyle=none](1,0)(2.5,0.5)(3,1.5)(2,3)(0.25,2.25)(0,1)
\psdots[dotsize=3pt](1,0)(2.5,0.5)(3,1.5)(2,3)(0.25,2.25)(0,1)(2.25,1.25)(1.75,2)
\psline(2,3)(1.75,2)(2.25,1.25)(3,1.5)(2,3)
\psline[linestyle=dotted,
dotsep=1.5pt,linewidth=1.7pt](1.75,2)(0,1)(0.25,2.25)(2,3)(1.75,2)
\uput[d](1,0){$\mu(o)$} \uput[dr](2.5,0.5){$\mu(p_1)$}
\uput[r](3,1.5){$\mu(q_1)$} \uput[u](2,3){$\mu(r)$}
\uput[ul](0.25,2.25){$\mu(q_3)$} \uput[dl](0,1){$\mu(p_3)$}
\uput[dr](2.25,1.25){$\mu(p_2)$} \uput[ul](1.75,2){$\mu(q_2)$}
\end{pspicture}}

\qquad

\subfigure[]{
\begin{pspicture}(0,-0.2)(3,3.5)\footnotesize
\pspolygon[fillstyle=solid,fillcolor=lightgray,
linestyle=none](1,0)(2.5,0.5)(3,2.25)(2,3)(0.25,2.25)(0,1)
\psline(2,3)(3,2.25)(1.25,0.75)(2.5,1.35)(2,3)
\psline[linestyle=dotted,
dotsep=1.5pt,linewidth=1.7pt](2,3)(2.5,1.35)(0,1)(0.25,2.25)(2,3)
\psdots[dotsize=3pt](1,0)(2.5,0.5)(3,2.25)(2,3)(0.25,2.25)(0,1)
\psdots[dotsize=3pt](1.25,0.75)(2.5,1.35)
\uput[d](1,0){$\mu(o)$} \uput[dr](2.5,0.5){$\mu(p_1)$}
\uput[r](3,2.25){$\mu(q_1)$} \uput[u](2,3){$\mu(r)$}
\uput[ul](0.25,2.25){$\mu(q_3)$} \uput[dl](0,1){$\mu(p_3)$}
\uput[d](1.25,0.75){$\mu(p_2)$} \uput[d](2.5,1.35){$\mu(q_2)$}
\end{pspicture}}

\qquad

\subfigure[]{
\begin{pspicture}(0,-0.2)(3,3.5)\footnotesize
\pspolygon[fillstyle=solid,fillcolor=lightgray,
linestyle=none](1,0)(2.5,0.5)(3,1.5)(2,3)(0.25,2.25)(0,1)
\psdots[dotsize=3pt](1,0)(2.5,0.5)(3,1.5)(2,3)(0.25,2.25)(0,1)(1.75,1)(2.25,2)
\psline(2,3)(2.25,2)(1.75,1)(3,1.5)(2,3)
\psline[linestyle=dotted,
dotsep=1.5pt,linewidth=1.7pt](2.25,2)(0,1)(0.25,2.25)(2,3)(2.25,2)
\uput[d](1,0){$\mu(o)$} \uput[dr](2.5,0.5){$\mu(p_1)$}
\uput[r](3,1.5){$\mu(q_1)$} \uput[u](2,3){$\mu(r)$}
\uput[ul](0.25,2.25){$\mu(q_3)$} \uput[dl](0,1){$\mu(p_3)$}
\uput[d](1.75,1){$\mu(p_2)$} \uput[dr](2.05,2){$\mu(q_2)$}
\end{pspicture}}
}
\end{center}
\caption{\label{figure: proof of convex} Possible configurations of $\Gamma$ of type 
Table \ref{table: graph shape}.(h)}
\end{figure}
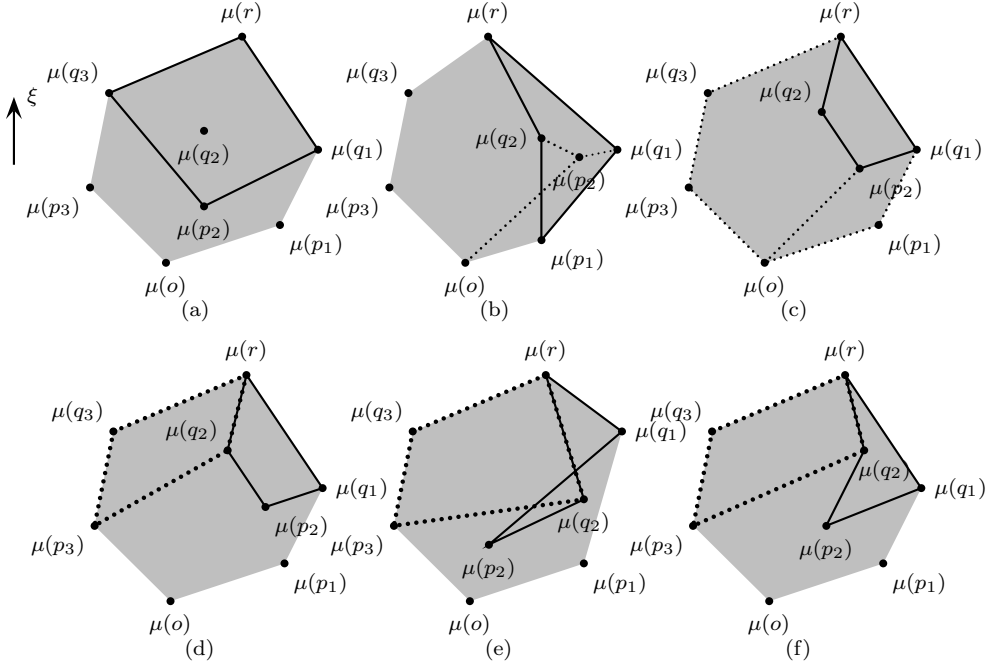

\subsection*{Case (ii): $\gamma_{p_1}$ contains $q_1$ and $q_2$}
First, $\gamma_{p_1}$ cannot be crossed since $(p_1, q_1)$ and $(q_1, r)$ 
are boundary edges. Suppose that $\gamma_{p_1}$ is
concave, see Figure \ref{figure: proof of convex}.(b).
Since $p_1$, $q_1$, and $r$ are boundary vertices of
$\gamma_{p_1}$, $q_2$ should be contained in the
interior of the convex hull $\Conv \{p_1, q_1,
r\}$ by Lemma \ref{lemma:
tetragon}. Then $q_2$ cannot be adjacent to $p_3$ by
Lemma \ref{lemma: interior}. Also, $q_2$ cannot be
adjacent to $o$ since $o$ is already adjacent to three
vertices $p_1$, $p_2$, and $p_3$. Thus $q_2$ is
adjacent to $p_2$ by the index increasing property.
Then $p_2$ should be in the interior of $\square
\gamma_{p_1}$ by Lemma \ref{lemma: interior} at $q_2$.
Moreover, by Lemma \ref{lemma: interior} again, $p_2$ cannot be adjacent to $q_3$.
Thus $p_2$ is adjacent to $q_1$ and this contradicts the three valency of $\Gamma$ 
at $p_3$. Therefore, $\gamma_{p_1}$ is convex. 

\subsection*{Case (iii): $\gamma_{p_2}$ contains $q_1$ and $q_2$}
Note that $p_3$ is adjacent to $q_2$ because $q_1$ is adjacent to $p_1$ and $p_2$, 
see Figure \ref{figure: proof of convex}.(c).
Suppose that $\gamma_{p_2}$ is crossed. Since the edge
$(q_1, r)$ is boundary, two line segments
$\overline{q_2 r}$ and
$\overline{p_2 q_1}$ should intersect by
Lemma \ref{lemma: tetragon}. Then this contradicts to
Lemma \ref{lemma_GKM_weight_pair_2} with respect to
the edge $(q_2, r)$, so $\gamma_{p_2}$ is not crossed, see Figure \ref{figure: proof of convex}.(e).

Next, suppose that $\gamma_{p_2}$ is concave.
By the index increasing property, $p_2$ should be below $q_1$ and $q_2$. Thus $q_2$ should be lying on the interior of $\Conv \{p_2, q_1, r\}$,
see Figure \ref{figure: proof of convex}.(f).
Then it contradicts Lemma \ref{lemma_GKM_weight_pair_2} with respect to the edge $(q_2, r)$.
Therefore, $\gamma_{p_2}$ is convex.

\subsection*{Case (iv): $\gamma_{p_3}$ contains $q_1$ and $q_2$}
Such case does not happen by the three valency at
$p_3$.
\end{proof}


\bigskip


\begin{thebibliography}{0}




\bibitem[AB]{AB} M. F. Atiyah and R. Bott, {\em The moment map and equivariant cohomology}, Topology \textbf{23} (1984), No.~1, 1--28.

\bibitem[BV]{BV} N. Berline and M. Vergne, {\em Classes caract\'{e}ristiques \'{e}quivariantes. Formule de localisation en cohomologie \'{e}quivariante}, C. R. Acad. Sci. Paris \textbf{295} (1982), 539--541.

\bibitem[At]{At}
M. F. Atiyah, {\em Convexity and commuting
Hamiltonians}, Bull. London Math. Soc. \textbf{14}
(1982), No. 1, 1--15.

\bibitem[Au]{Au}
M. Audin, {\em Torus actions on symplectic manifolds,}
Second revised edition. Progress in Mathematics \textbf{93}, Birkh\"{a}user Verlag, Basel, 2004.

\bibitem[Cho1]{Cho1}
Y. Cho, {\em Hard Lefschetz property of symplectic structures on compact K\"{a}hler manifolds}, Trans. Amer. Math. Soc. \textbf{368} (2016), No. 11,  8223--8248.

\bibitem[Cho2]{Cho2}
Y. Cho, {\em Unimodality of Betti numbers for Hamiltonian circle actions with index-increasing moment maps}, Internat. J. Math. \textbf{27}  (2016), No. 5,
1650043, 14pp.

\bibitem[CK1]{CK1}
Y. Cho, M. K. Kim, {\em Unimodality of the Betti
numbers for Hamiltonian circle action with isolated
fixed points}, Math. Res. Lett. \textbf{21} (2014),
No. 4, 691--696.

\bibitem[CK2]{CK2}
Y. Cho, M. K. Kim, {\em Hamiltonian circle action with
self-indexing moment map}, Math. Res. Lett.  \textbf{23} (2016), 719-732.

\bibitem[De]{De}
T. Delzant, {\em Hamiltoniens p\'{e}riodiques et image convex de l'application moment,} Bull. Soc. Math. France \textbf{116} (1988), 315--339.

\bibitem[Fr]{Fr} T. Frankel, {\em Fixed points and torsion on Kahler manifolds,}  Ann. of Math. \textbf{70} (1959), 1-8.

\bibitem[Go]{Go}
R. Gompf, {\em A new construction of symplectic
manifolds}, Ann. of Math. (2) \textbf{142} (1995), No.
3, 527--595.

\bibitem[GKM]{GKM}
M. Goresky, R. Kottwitz, and R. MacPherson, {\em
Equivariant cohomology, Koszul duality, and the
localization theorem,} Invent. Math. \textbf{131}
(1998), No. 1, 25--83.


\bibitem[GS]{GS}
V. Guillemin and S. Sternberg, {\em Convexity properties
of the moment mapping}, Invent. Math. \textbf{67}
(1982), No. 3, 491--513.

\bibitem[GS2]{GS2}
V. Guillemin and S. Sternberg, {\em Supersymmetry and equivariant de Rham theory}, Springer-Verlag, 1999.

\bibitem[GT]{GT}
R. Goldin, S. Tolman, {\em Towards Generalizing Schubert
Calculus in the Symplectic Category}, J. Symplectic Geom. \textbf{7} (2009), No. 4, 449-473.


\bibitem[GZ]{GZ}
V. Guillemin, C. Zara, {\em Combinatorial formulas for
products of Thom classes,} Geometry, mechanics,
and dynamics, Springer, New York, 2002, 363--405.




\bibitem[JHKLM]{JHKLM}
L. Jeffrey, T. Holm, Y. Karshon, E. Lerman, E.
Meinrenken, {\em Moment maps in various geometries},
available online at
http://www.birs.ca/workshops/2005/05w5072/report05w5072.pdf

\bibitem[Ka]{Ka} Y. Karshon, {\em Periodic Hamiltonian flows on four-dimensional manifolds,}
Mem. Amer. Math. Soc. \textbf{141} (1999), No. 672.

\bibitem[Ka2]{Ka2} Y. Karshon, {\em Hamiltonian torus action,}
Geom. Phys. Lecture notes in Pure and Applied Mathematics Series \textbf{184}, Marcel Dekker (1996), 221-230.

\bibitem[Ki]{Ki}
F. C. Kirwan, {\em Cohomology of quotients in
symplectic and algebraic geometry}, Princeton
University Press, 1984.

\bibitem[Km]{Km}
M. K. Kim, {\em Frankel's theorem in the
symplectic category}, Trans. Amer. Math. Soc.
\textbf{358} (2006), No. 10, 4367--4377.

\bibitem[Lu]{Lu}
S. Luo, {\em The hard Lefschetz property for
Hamiltlnian GKM manifolds}, J. Algebr. Comb.
\textbf{40}, No. 1, 45--74.


\bibitem[Ma]{Ma} O. Mathieu, {\em Harmonic cohomology classes of symplectic manifolds}, Comment. Math. Helv. \textbf{70} (1995), No. 1, 1--9.

\bibitem[Mc]{Mc} P. McMullen, {\em On simple polytopes,} Invent. Math. \textbf{113} (1993), 419-444.

\bibitem[Mo]{Mo}
D. Morton, {\em GKM manifolds with low Betti numbers}
Ph.D. Thesis, University of Illinois at
Urbana-Champaign, 2011.

\bibitem[ST]{ST}
S. Sabatini, S. Tolman, {\em New techniques for
obtaining Schubert-type formulas for Hamiltonian
manifolds}, J. Symplectic Geom. \textbf{11} (2013),
No. 2, 179--230 .

\bibitem[T]{T}
S. Tolman, {\em Examples of non-K\"{a}hler Hamiltonian
torus actions}, Invent. Math. \textbf{131} (1998),
No.~2, 299--310.


\bibitem[We]{We}
M. J. Wenninger, {\em Dual Models}, Cambridge
University Press, 1983.

\bibitem[Wo]{Wo}
C. Woodward, {\em Multiplicity-free Hamiltonian
actions need not be K\"{a}hler}, Invent. Math.
\textbf{131} (1998), 311--319.

\bibitem[Wo2]{Wo2}
C. Woodward, {\em Multiplicity-free Hamiltonian actions need not be K\"{a}hler,} arXiv:dg-ga/9506009v1.



\end{thebibliography}
\end{document}